%% file: Sigmetrics18_DBS.tex
\documentclass[format=acmsmall, review=false, screen=true]{acmart}

\usepackage{booktabs} % For formal tables
\usepackage{amsmath, amsthm, amsfonts, amssymb, graphicx, epsfig,epstopdf} 

\usepackage[ruled]{algorithm2e} % For algorithms

\SetAlFnt{\small}
\SetAlCapFnt{\small}
\SetAlCapNameFnt{\small}
\SetAlCapHSkip{0pt}
\IncMargin{-\parindent}

\setcopyright{acmlicensed}
\newtheorem{assumption}{Assumption}

\newtheorem*{remark*}{Remark}
\newtheorem{notation}{\bf{Notation}}
\newtheorem*{notation*}{\bf{Notation}}

\newcommand{\xbf}{\mathbf{x}}
\newcommand{\vbf}{\mathbf{v}}

\newcommand{\ybf}{\mathbf{y}}
\newcommand{\gbf}{\mathbf{g}}
\newcommand{\hbf}{\mathbf{h}}

\newcommand{\xbar}{\bar{x}}
\newcommand{\xbarbf}{\mathbf{\xbar}}

\newcommand{\vbar}{\bar{v}}
\newcommand{\1}{\mathbf{1}}
\newcommand{\abf}{\mathbf{a}}
\newcommand{\zetab}{\bar{\zeta}}
\newcommand{\zetabf}{\boldsymbol{\zeta}}

\newcommand{\Xbf}{\mathbf{X}}
\newcommand{\Vbf}{\mathbf{V}}
\newcommand{\Xbarbf}{\bar{\Xbf}}
\newcommand{\A}{\mathbf{A}}
\newcommand{\I}{\mathbf{I}}
\newcommand{\Ybf}{\mathbf{Y}}
\newcommand{\Gbf}{\mathbf{G}}
\newcommand{\Hbf}{\mathbf{H}}
\newcommand{\Gcal}{\mathcal{G}}
\newcommand{\Vcal}{\mathcal{V}}
\newcommand{\Ecal}{\mathcal{E}}
\newcommand{\Xcal}{\mathcal{X}}
\newcommand{\Tcal}{\mathcal{T}}
\newcommand{\Pcal}{\mathcal{P}}
\newcommand{\Nical}{\mathcal{N}_i}
\newcommand{\Rset}{\mathbb{R}}

\begin{document}

\title{On the Convergence Rate of Distributed Gradient Methods for Finite-Sum Optimization under Communication Delays}

\author{Thinh T. Doan}
\affiliation{
  \institution{Coordinated Science Lab, University of Illinois, Urbana-Champaign}
  \city{Urbana}
  \state{IL}
  \postcode{61801}
  \country{USA}}
\email{ttdoan2@illinois.edu}

\author{Carolyn L. Beck}
\affiliation{
  \institution{Coordinated Science Lab, University of Illinois, Urbana-Champaign}
  \city{Urbana}
  \state{IL}
  \postcode{61801}
  \country{USA}}
\email{beck3@illinois.edu}

\author{R. Srikant}
\affiliation{
  \institution{Coordinated Science Lab, University of Illinois, Urbana-Champaign}
  \city{Urbana}
  \state{IL}
  \postcode{61801}
  \country{USA}}
\email{rsrikant@illinois.edu}

%\author{Thinh T. Doan \qquad Carolyn L. Beck \qquad R. Srikant \thanks{The authors are with the Coordinated Science Laboratory, University of Illinois at  Urbana-Champaign, IL, USA
%        {\tt\small \{ttdoan2,beck3,rsrikant\}@illinois.edu}}}
%\date{}
%%\email{rsikant@illinois.edu}

\begin{abstract}
Motivated by applications in machine learning and statistics, we study distributed optimization problems over a network of processors, where the goal is to optimize a global objective composed of a sum of local functions. In these problems, due to the large scale of the data sets, the data and computation must be distributed over processors resulting in the need for distributed algorithms. In this paper, we consider a popular distributed gradient-based consensus algorithm, which only requires local computation and  communication. An important problem in this area is to analyze the convergence rate of such algorithms in the presence of communication delays that are inevitable in distributed systems.  We prove the convergence of the gradient-based consensus algorithm in the presence of uniform, but possibly arbitrarily large, communication delays between the processors. Moreover, we obtain an upper bound on the rate of convergence of the algorithm as a function of the network size, topology, and the inter-processor communication delays.
\end{abstract}

%\begin{CCSXML}
%<ccs2012>
% <concept>
%  <concept_id>10010520.10010553.10010562</concept_id>
%  <concept_desc>Machine learning~Distributed Optimization</concept_desc>
%  <concept_significance>500</concept_significance>
% </concept>
% <concept>
%  <concept_id>10010520.10010575.10010755</concept_id>
%  <concept_desc>Computer and communication networks~Distributed algorithms</concept_desc>
%  <concept_significance>300</concept_significance>
% </concept>
%</ccs2012>
%\end{CCSXML}
%
%\ccsdesc[500]{Machine learning~Distributed optimization}
%\ccsdesc[300]{Computer and communication networks ~Distributed algorithms}


\maketitle

\input{intro.tex}

\input{formulation.tex}

\input{alg_delay.tex}

\input{results_tv.tex}

\input{simulation.tex}

\section{Proofs of Main Results}\label{sec:Proofs}
\input{uniform_analysis.tex}

\section{Concluding Remarks}\label{sec:conclusion}
In this paper we have studied a continuous-time distributed gradient-based consensus algorithm for network optimization problems, with the focus on uniform communication delays. We provided an explicit analysis on the rate of convergence of the algorithm as a function of the network size, topology, and communication delays, specifically the convergence time of the algorithm grows as a cubic function of the delays. We also simulate the performance of the distributed gradient algorithm for the delay-free case and with uniform delays for different network sizes, and compare with the performance of distributed dual averaging. Our simulation results suggest that distributed gradient outperforms dual averaging in both cases.

One interesting question left open in this paper is the study of asynchronous distributed gradient algorithms, that is, when communications delays are different at different nodes and perhaps change with time. In this more general case, it would be interesting to investigate whether an upper bound on the time-varying heterogeneous delays can be helpful in obtaining convergence results. In particular, a possible topic of future research would be to determine if one can obtain bounds on the error in the objective function by using an upper bound on the delays, along with our current results.

%Moreover, it is often the case that the training data is corrupted by noise and the network may suffer communication failures, which leads to the question of the robustness of the algorithms under communication delays, communication failures, and noisy updates. Finally, our simulations suggest that in the case of uniform communication delays, distributed dual averaging methods have a convergence rate similar to the rate of distributed gradient methods.  Proving such a result is another potential research direction.

\newpage
\bibliographystyle{ACM-Reference-Format}
\bibliography{refs}

%\bibliographystyle{plain}
%\bibliography{refs}

%\appendix
\input{appendix.tex}
\end{document}

%% file: intro.tex
%!TEX root = Sigmetrics18_DBS.tex

\section{Introduction}

There has been much recent interest in large-scale optimization problems, especially in machine learning and statistics. Due to the explosion in the size of data sets, it is important to be able to solve such problems efficiently. In addition, very often large data sets, on the order of terabytes, cannot be stored or processed on one single processor. As a result, both the data and computation must be distributed over a network of processors, necessitating the development  of distributed algorithms. Moreover, the computation and communication in these algorithms should be efficient enough so that network latencies do not offset the computational gains.

In this paper, we study distributed algorithms for optimization problems that are defined over a network of nodes\footnote{The terms nodes and processors will be used interchangeably.}, while explicitly accounting for network delays, one of the most critical issues in distributed systems. The objective function is defined by a sum of local functions where each function is known by only one node. Problems of this nature arise in a variety of application domains within the information sciences and engineering. A standard example from statistical machine learning \cite{Li2014} is the problem of minimizing an average loss function over large training data. The data is distributed across a network of processors, where each processor computes the empirical loss over a local subset of data. The processors, therefore, must communicate to determine parameters that minimize the loss over the entire data set. Distributed algorithms for these problems have received a surge in interest in recent years. In particular, there are three widely-studied algorithms for distributed optimization:

\begin{enumerate}
\item \textit{Alternating direction method of multipliers (ADMM)}: This method has a provably fast convergence rate, i.e., an exponential convergence rate under assumptions of strong convexity and smoothness of objective functions; see for example the work in \cite{Giannakis2010, Boyd2011, Shi2014, Ozdaglar2014, Ozdaglar2013}. However, the computations of ADMM are not truly parallelizable. The algorithm is often said to have a distributed implementation, which means that different processors compute different variables, but the updates of these variables must be performed sequentially.

\item \textit{Distributed dual averaging}: In this algorithm, processors maintain estimates of variables and gradient-like quantities, which are exchanged in a  truly parallel fashion. However, dual averaging has a slower convergence rate than ADMM; see for example, the work in  \cite{Duchi2012, Rabbat2012a, Rabbat2012b, Rabbat2012c}.

\item \textit{Distributed gradient algorithms}: These algorithms are the most popular and well-studied since they have the benefits of both ADMM and dual averaging; see for example, the work in \cite{Touri2015, Nedic2009, Nedic2010a, Shi2015, Nedic-others2016, NaLi2016,Cortes2014}). In particular, distributed gradient algorithms are parallelizable like dual averaging and have fast convergence rates like ADMM. Moreover, the computation cost of each iteration is smaller than either dual averaging or ADMM.
\end{enumerate}
In this paper, we study distributed gradient methods because of the advantages stated above. In particular, we focus on the convergence in the presence of inter-processor communication delays, which has been identified as an important problem in \cite{Palomar2009} (see chapter 10). Communication delay, which is one of the most fundamental issues in distributed systems, has been studied in other contexts, such as distributed dual averaging \cite{Rabbat2012c}. The analysis in \cite{Rabbat2012c} is based on adding fictitious nodes corresponding to the number of time delay steps, thus requiring a modification of the true network topology. As a result, the influence of the delays on the convergence rate for the original network topology is not clear. Convergence under delays are also considered in distributed consensus algorithms \cite{Blondel05,Nedic2010b,Munz2011,Rabbat2012,Charalambous15}, which are special cases of distributed gradient algorithms. However, these results do not apply to the general distributed algorithms considered here. Our goal in this paper, therefore, is to address this open problem of proving convergence and obtaining convergence rates for distibuted gradient algorithms with inter-processor communication delays.

 \textbf{Main Contributions}. The main contribution of this paper is to derive the convergence rate of distributed gradient algorithms under uniform communication delays between nodes. In particular, we first show that under some appropriate choice of stepsizes the nodes' estimates asymptotically converge to the solution of the problem, implying that the impact of communication delays is asymptotically negligible. This step allows us to study the rate of convergence of the algorithm, i.e., the convergence occurs at rate $\mathcal{O}\Big(\frac{n\tau^3\ln(t)}{(1-\gamma)^2\sqrt{t}}\Big)$, where $n$ is the number of processors, $t$ is the time variable, and $\tau$ is the delay constant. In addition, $\gamma$ is a constant in $(0,1)$ that depends on $\sigma_2$, the spectral properties of network connectivity of the processors. We note that such an explicit formula for the convergence rate is not available for dual averaging methods. As remarked, the existing analysis in distributed optimization literature cannot be extended to show this result. We, therefore, introduce a new approach by considering a new candidate Lyapunov functional, which takes into account the impact of delays. Finally, while we do not analyze dual averaging methods in the presence of delays, we provide simulation results comparing it to distributed gradient methods, which indicate that distributed gradient methods perform significantly better.

The remainder of this paper is organized as follows. We give a formal statement of distributed optimization problems in Section \ref{sec:ProbForm}. We then study distributed gradient algorithms for the uniform delay case in Section \ref{sec:alg_delay} and present their convergence results in Section \ref{sec:delay}. In Section \ref{sec:simulations} we compare the performances of distributed gradient methods and dual averaging methods by simulations for both the delay-free and uniform delay cases. The proofs of our main results in Sections \ref{sec:delay} are given in Section \ref{sec:Proofs}. Finally, we conclude this paper with some discussion of potential future extensions in Section \ref{sec:conclusion}.

\begin{notation}
We use boldface to distinguish between vectors $\xbf$ in $\mathbb{R}^n$ and scalars $x$ in $\mathbb{R}$. Given any vector $\xbf\in\mathbb{R}^n$, we write $\xbf=(x_1,x_2,\ldots,x_n)$ and let $\|\xbf\|_2$ denote its Euclidean norm. Given a vector $\xbf$ and a set $\Xcal$ we write the projection of $\xbf$ on $\Xcal$ as $\Pcal_\Xcal[\xbf]$. Finally we denote by $\mathbf{1}$ and $\I$ a vector whose entries are $1$ and the identity matrix, respectively.
\end{notation}

%% file: formulation.tex
%!TEX root = Sigmetrics18_DBS.tex

\section{Problem formulation}\label{sec:ProbForm}
In this paper, we consider an optimization problem where the objective function is distributed over a  network of $n$ nodes. In particular, let $\Gcal=(\Vcal,\Ecal)$ be an undirected graph over the vertex set $\Vcal=\{1,\ldots.n\}$ with the edge set $\Ecal = (\Vcal\times \Vcal)$. Associated with each node $i\in \Vcal$ is a convex function $f_i:\mathbb{R}^d\rightarrow\mathbb{R}$. The goal of the network is to solve the following minimization problem:
\begin{align}
\text{minimize } \sum_{i=1}^n f_i(\xbf) \text{ over }\xbf\in\Xcal,\label{prob:obj}
\end{align}
where $\Xcal\subseteq\Rset^d$ is compact, convex, and known by the nodes. We assume no central coordination between the nodes and since each node knows only a local function $f_i$, the nodes are required to cooperatively solve the problem. We are interested in studying distributed consensus-based methods for problem \eqref{prob:obj} implying that each node $i$ maintains its own parameter estimate $\xbf_i\in\mathbb{R}^d$ which is used to estimate the solution of \eqref{prob:obj}. The nodes are only allowed to exchange their estimates with their neighbors through communication constraints imposed by a graph $\Gcal$: in particular, node $i$ can communicate directly only with its neighbors $j\in\mathcal{N}_i$ where $\Nical:=\{j\in\Vcal | (i,j)\in \Ecal\}$ is the set of node $i$'s neighbors. The goal is to asymptotically drive the nodes' estimates $\xbf_i$ to $\xbf^*$, a solution of \eqref{prob:obj}.

A concrete motivating example for this problem is distributed linear regression problems solved over a network of processors. Regression problems involving massive amounts of data are common in machine learning applications. Each function $f_i$ is the empirical loss over the local data stored at processor $i$. The objective is to minimize the total loss over the entire dataset. Due to the difficulty of storing the enormous amount of data at a central location, the processors perform local computations over the local data, which are then exchanged to arrive at the globally optimal solution. Distributed gradient methods are a natural choice to solve such problems since they have been observed to be both fast and easily parallelizable in the case where the processors can exchange data instantaneously. The goal of this paper is to show that the algorithm continues to be convergent in the presence of delays, and to derive expressions for the convergence rate as a function of the delays. Another possible application of the model is the problem of estimating the radio frequency in a wireless network of sensors where the goal is to cooperatively estimate the radio-frequency power spectrum density through solving a regression problem  \cite{Giannakis2010}. In this application, each function $f_i$ is the empirical loss over the local data measured by the sensors, which are scattered across a large geographical area. The objective function is the total loss over the entire measured data, which is the sum of $f_i$. Due to privacy concerns, the sensors may not be willing to share their measurements, but only their own estimates. Thus, distributed consensus-based methods seem to be a proper choice for this problem.

We conclude this section with additional notation and assumptions which facilitate our development given later. We make the following assumptions throughout the paper.

\begin{assumption}\label{assump:convexity}
The functions $f_i$ are convex and differentiable.
\end{assumption}
\begin{assumption}\label{assumption:ConnectedGraph}
The graph $G$ is undirected and connected.
\end{assumption}
Under Assumption \ref{assump:convexity} and since the set $\Xcal$ is compact, there exists a point $\xbf^*$ which solves problem \eqref{prob:obj}. However, $\xbf^*$ may not be unique.  We will use $\mathcal{X}^*$ to denote the set of optimal solutions to problem \eqref{prob:obj}.  Moreover, given a solution $x^*\in\mathcal{X}^*$ we denote $f^* = \sum_{i=1}^n f_i(x^*)$. Under Assumption \ref{assump:convexity} it is obvious that the functions $f_i$ are Lipschitz continuous, which we present below as a Proposition for future reference.

\begin{proposition}\label{prop:Lipschitz}
Let Assumption \ref{assump:convexity} hold. Then each function $f_i$ is Lipschitz continuous, i.e., there exists a positive constant $C_i$ such that 
\begin{align}
|f_i(\xbf) - f_i(\ybf)| \leq C_i \|\xbf-\ybf\|_2,\quad \forall \xbf,\ybf \in \Xcal,\; \forall i\in\Vcal.\label{prop_eq:Lipschitz}
\end{align}
\end{proposition}

Given a vector $\xbf\in\Xcal$ we denote by $\mathcal{D}_{\Xcal}(\xbf)$ the set of feasible directions of $\xbf$ in $\Xcal$, i.e., 
\begin{align}
\mathcal{D}_{\mathcal{X}}(\xbf) = \{\ybf\in\mathbb{R}^d\;|\;\exists\; \theta > 0 \text{ s.t. } \xbf + \theta \ybf \in \mathcal{X} \}. \label{def:feasible_set}
\end{align}
In the sequel we use the following results from \cite{Bertsekas2003}.

\begin{proposition}[Proposition $4.6.2$ \cite{Bertsekas2003}] \label{prop:TangentCone}
Let $\mathcal{X}$ be a closed convex set. Then the tangent cone $\mathcal{T}_{\mathcal{X}}(\xbf)$ at $\xbf\in\mathcal{X}$ is closed, convex, and $\mathcal{T}_{\mathcal{X}}(\xbf) = cl(\mathcal{D}_{\Xcal}(\xbf))$, where $cl(\mathcal{D}_{\Xcal}(\xbf))$ is the closure of $\mathcal{D}_{\Xcal}(\xbf)$.
\end{proposition}

Finally, for ease of exposition, in the rest of this paper we consider problem \eqref{prob:obj} when the variable $\xbf$ is a scalar, i.e., $d=1$. Extensions for the case $d>1$ are presented in the appendix.

%% file: alg_delay.tex
\section{Distributed Gradient Methods under Communication Delays}\label{sec:alg_delay}
Discrete-time distributed gradient methods were studied and first analyzed rigorously in \cite{Nedic2009,Nedic2010a} for the case of no communication delay; in this framework each node $i\in\Vcal$ maintains a variable $x_i\in\mathbb{R}$ updated as,
\begin{align}
&x_i(k+1) = \mathcal{P}_{\mathcal{X}}\left[\!\begin{array}{c}
\sum_{j\in\mathcal{N}_i}a_{ij}x_j(k) -\alpha(t)f_i'(x_i(k))\end{array}\!\right],\label{DGM:discrete-time}
\end{align}
where $\alpha(t)$ is some sequence of positive stepsizes and $a_{ij}$ is some positive constant. In this paper we focus on the continuous-time version of \eqref{DGM:discrete-time} under the impact of uniform communication delays between nodes. In particular, we assume that at any time $t\geq 0$ node $i$ only receives a delayed value $x_j(t-\tau)$ of $x_j(t)$ from node $j$, where $\tau$ is a constant representing the time delay of communication between nodes. Each node $i$ (for all $i \in \Vcal$) then uses these values to update its estimate as formally stated in \eqref{delay:xidot}, where $\Tcal_{\Xcal(x_i(t))}$ is the tangent cone of $\Xcal$ at $x_i(t)$, $\beta$ is some postive constant, and $\alpha(t)$ is a sequence of positive stepsizes. The conditions of $\beta$ and $\alpha(t)$ to guarantee convergence of the algorithm will be explicitly given later. In addition, the initial conditions, $\phi_i(t)$, are assumed to be continuous functions of time. Thus, the estimates $x_i(t)$ are now functionals since they are functions of $\phi_i(t)$. We assume that the delays are uniform across agents, represented by the positive constant $\tau$.

This update has a simple interpretation: at any time $t\geq 0$, each node $i$ first combines its estimate $x_i(t)$ with the weighted, delayed values received from its neighbors $j\in\Nical$, with the goal of seeking consensus on their estimates. Each node then moves along the gradient of its respective objective function to update its estimate, pushing the consensus point toward the optimal set $\Xcal^*$. The projection on the tangent cone $\Tcal_{\Xcal(x_i(t))}$ guarantees that $x_i(t)\in\Xcal$ for all $t\geq 0$. Here the positive constant $a_{ij}$ represents the weight which node $i$ assigns to the value $x_j$ received from node $j$. Moreover, the nodes use the positive constant $\beta$, which is inversely proportional to the delay constant $\tau$, to control the speed of their updates. The distributed gradient algorithm with communication delays is formulated in Algorithm \ref{alg:delay}.

In the sequel, we denote by $\A$ the $n \times n$ weighted adjacency matrix corresponding to the graph $\Gcal$, whose $(i,j)$-th entries are $a_{ij}$. We make an assumption on $\A$ which is standard in the consensus literature to guarantee the convergence of the nodes' estimates to a consensus point. The assumption given below also imposes a constraint on the communication between the nodes in Algorithm \ref{alg:delay} in which the nodes are only allowed to exchange messages with neighboring nodes, i.e., those are connected to them, as defined by $\Gcal$.
\begin{assumption}\label{assump:doublystochastic}
$\A$ is a doubly stochastic matrix, i.e., $\sum_{i=1}^n a_{ij} = \sum_{j=1}^n a_{ij} = 1$. Moreover, $\A$ is assumed to be irreducible and aperiodic. Finally, the weights $a_{ij} > 0$ if and only if $(i,j)\in \Ecal$ otherwise $a_{ij} = 0$.
\end{assumption}
We note that the assumption on the irreducibility of $\A$ can be satisfied when $\Gcal$ is connected. In addition, the aperiodicity of $\A$ is guaranteed when at least one of its diagonal $a_{ii}$ is strictly positive. Finally, the double stochasticity of $\A$ is essential to the distributed consensus averaging problem \cite{Nedic-others2008}, a special case of problem \eqref{prob:obj}. There has been some work in which this assumption is relaxed to just stochasticity of $\A$, however; additional assumptions on the problem are then imposed; see for example, push-sum protocols recently studied in \cite{Nedic2015}. 

%An extension to cover this case is interesting but it is beyond the scope of this paper, where we focus on studying the convergence rate of distributed gradient algorithm under communication delays. 
\vspace{-0.2cm}

%\RestyleAlgo{boxruled}
\begin{algorithm}
\caption{Distributed Gradient Algorithm With Delays}
1. \textbf{Initialize}: Each node $i$ is initiated with a point $x_i(t) = \phi_i(t)\in\Xcal, \quad t\in[-\tau,0]$.\\
2. \textbf{Iteration}: For $t\geq 0$ each node $i\in\Vcal$ executes
\vspace{0mm}
{\small
\begin{align}
\dot{x}_i(t) = \Pcal_{\Tcal_{\Xcal(x_i(t))}}\left[-  \beta x_i(t) + \beta \sum_{j=1}^n a_{ij}x_j(t-\tau) - \alpha(t) f_i'(x_i(t))\right]\label{delay:xidot}
\end{align}}
\label{alg:delay}
\end{algorithm}\vspace{-0.5cm}

%% file: results_tv.tex
%!TEX root = Sigmetrics18_DBS.tex

\section{Convergence Results}\label{sec:delay}
The focus of this section is to analyze the performance of distributed gradient methods under communication delays given in Algorithm \ref{alg:delay}. In particular, we provide a rigorous analysis which establishes the convergence rate of Algorithm \ref{alg:delay}. The main steps of the analysis are as follows. 

We first show that the distances between the estimates $x_i(t)$ to their average $\xbar(t)$ asymptotically converge to zero. We then study the convergence rate of Algorithm \ref{alg:delay}, where we utilize the standard techniques used in the centralized version of subgradient methods. The key idea of this step is to introduce a candidate Razumikhin-Krasovskii Lyapunov functional, which takes into account the impact of delays on the system. By using this function, we can show that the impact of delays is asymptotically negligible. In particular, we show that if each node maintains a variable $z_i(t)$ to compute the time-weighted averages of the estimates $x_i(t)$ and if the stepsize decays with rate $\alpha(t) = 1/\sqrt{t}$, the algorithm achieves an asymptotic convergence to the optimal value estimated on the variable $z_i(t)$ at a rate $\mathcal{O}\Big(\frac{n\tau^3\ln(t)}{(1-\gamma)^2\sqrt{t}}\Big)$, where $\gamma =\sigma_2 e^{\beta\tau} \in (0,1)$ and $\beta \in (0,\frac{\ln(1/\sigma_2)}{\tau})$. Here  $\sigma_2$ represents the algebraic connectivity of the graph $\Gcal$. 

We start our analysis by first introducing more notation. Given a vector $\mathbf{x}\in\mathbb{R}^n$ we denote its average as  $\bar{x}$, i.e.,
\begin{align}
\bar{x} = \frac{1}{n}\mathbf{1}^T\xbf =  \frac{1}{n}\sum_{i=1}^n x_i.\nonumber
\end{align}
For convenience, we use the following notation,
\begin{align*}
&F(\mathbf{x}) \triangleq \sum_{i=1}^n f_i(x_i),\quad\nabla F(\xbf(t)) \triangleq [f_1'(x_1),\ldots,f_n'(x_n)]^T,\quad C \triangleq \sum_{i=1}^n C_i.
%\quad\zetabf(t) \triangleq \zbf(t) - \mathcal{P}_{\Tcal_{\Xcal(\xbf(t))}}[\zbf(t)],
\end{align*}
%where $\mathcal{P}_{\Tcal_{\Xcal(\xbf(t))}}$ is component-wise projection. Here $\zetabf$ denotes the error due to projection.

We denote by $\sigma_2$ the second largest singular value of $\A$, i.e., $\sigma_2$ is the square root of the second largest eigenvalue of $\A^T\A$. Since $\A$ is doubly stochastic we have $\A^T\A$ is also doubly stochastic. In addition, $\A$ also satisfies Assumption \ref{assump:doublystochastic}. Thus, by the Perron-Frobenius theorem \cite{HJbook} we have $\sigma_2\in(0,1)$. 

Finally, without loss of generality we consider $\Xcal = [a,b]$ for some real numbers $a\leq b\in\Rset$. The multi-dimensional case of $\Xcal$ is presented in the Appendix. This simplification will allow us to write explicitly the projection on the tangent cone in \eqref{delay:xidot}. In particular, given a real number $v$ we denote $v^{+} = \max(0,v)$, the positive part of $v$. Similarly, we denote $v^{-} = \max(0,-v)$, the negative part of $v$. The update in \eqref{delay:xidot} can now be rewritten as 
\begin{align}
v_i(t) &= -\beta x_i(t) + \beta \sum_{j=1}^n a_{ij}x_j(t-\tau) - \alpha(t) f_i'(x_i(t))\label{analysis:viUpdate}\\
\dot{x_i}(t) &= \Pcal\left(v_i(t)\right) = \left\{
\begin{array}{ll}
v_i(t) & \text{if }\; x_i(t) \in (a,b)\\
v_i^{+}(t) & \text{if }\; x_i(t) = a\\
-v_i^{-}(t) & \text{if }\; x_i(t) = b
\end{array}\right.\label{analysis:xiUpdate}
\end{align}
Given $v_i\in\Xcal$ we denote by $\zeta_i$ the error due to projection of $v_i$ to $\Tcal_{\Xcal(x_i)}$, i.e., $\zeta_i(v_i) = v_i - \Pcal\left(v_i\right).$
Using this notation and $\A$ equations \eqref{analysis:viUpdate} and \eqref{analysis:xiUpdate} can be rewritten in vector form as
\begin{align}
&\vbf(t)
= -\beta\xbf(t) + \beta \A\xbf(t-\tau) - \alpha(t) \nabla F(\xbf(t)),\label{analysis:vUpdate}\\
&\dot{\xbf}(t)
= \mathcal{P}(\vbf(t))= \vbf(t) - \zetabf(\vbf(t)),\label{analysis:xUpdate}
\end{align}
where $\mathcal{P}(\vbf(t))$ denotes the component-wise projection. Moreover, we have
\begin{align}
&\vbar(t) = -\beta\xbar(t) + \beta\xbar(t-\tau) - \frac{\alpha(t)}{n}\sum_{i=1}^n f_i'(x_i(t))\label{analysis:vbarUpdate}\\
&\dot{\bar{x}}(t)
= \bar{z}(t)-\zetab(\vbf(t)).\label{analysis:xbarUpdate}
\end{align}
As remarked, the first step in our analysis is to show the asymptotic convergence of $\|\mathbf{x}(t)-\xbar(t)\1\|_2$ to zero under some appropriate choice of stepsizes. The following Lemma, which will be essential for our analysis later, is an important facet of this result. 

%We note that without the communication delays, the convergence of $\|\mathbf{x}(t)-\xbar(t)\1\|_2$ to zero is quite obvious; see the analysis given in the Appendix. However, under communication delays the estimates \eqref{analysis:xUpdate} are functional, i.e., $\xbf(t)$ depends on the time interval $[t-\tau,t]$ for all $t\geq 0$. This makes the analysis of this Lemma more challenging compared to the delay-free case \cite{Nedic2010a}. We, therefore, utilize techniques from the $Gr\ddot{o}nwall$-$Bellman$ Inequality \cite{KHALIL02}, to show this result.   

\begin{lemma}\label{lem:pertbAverdelay}
Suppose Assumptions \ref{assump:convexity}-- \ref{assump:doublystochastic} hold. Let the trajectories of $x_i(t)$ be updated by Algorithm \ref{alg:delay}. Let $\{\alpha(t)\}$ be a given positive scalar sequence with $\alpha(0) = 1$. Moreover, let $\beta \in (0,\frac{\ln(1/\sigma_2)}{\tau})$ and $\gamma =\sigma_2 e^{\beta\tau} \in (0,1)$. Then 
\begin{itemize}
\item[(1)] For all $t\geq 0$  we have
\begin{align}
\|\xbf(t)-\bar{x}(t)\mathbf{1}\|_2 \leq \mu(t) + \beta\sigma_2\int_{0}^{t}e^{-\beta(1-\gamma)(t-u)}\mu(u-\tau)du,\label{DPAlem:AverIneq}
\end{align}
where
\begin{align}
&\mu(t) =   \frac{\|\xbf(0)\|_2+2C}{\beta}e^{-\beta t/2} + \frac{2C\alpha(t/2)}{\beta}.\label{DPAlem:lambda}
\end{align}
\item[(2)]
If $\{\alpha(t)\}$  is a non-increasing positive scalar sequence such that $\lim_{t\rightarrow\infty}\alpha(t) = 0$ then we have
\begin{align}
\lim_{t\rightarrow\infty} |x_i(t) - \bar{x}(t)| = 0\quad \text{for all } i=1,2\ldots,n.\label{DPAlem:AsympConv}
\end{align}
\item[(3)] Further we have
\begin{align}
\int_{0}^t \alpha(u)\|\xbf(u)-\xbar(u)\mathbf{1}\|_2du\leq \frac{8\left(\|\xbf(0)\|_2+2C\right)e^{\beta\tau/2}}{\beta^3(1-\gamma)^2}+ \frac{4C}{\beta^2(1-\gamma)}\int_{0}^{t}\alpha^2(\gamma u/4-\tau)du.\label{DPAlem:Upperbound}
\end{align}
\end{itemize}
\end{lemma}

\begin{proof}[Proof sketch]
The main idea in the proof of Lemma \ref{lem:pertbAverdelay} is to show \eqref{DPAlem:AverIneq}. The analysis of \eqref{DPAlem:AsympConv} and \eqref{DPAlem:Upperbound} are consequences of \eqref{DPAlem:AverIneq} with the given assumptions on stepsizes and proper algebraic manipulations. We, therefore, provide here the key steps for the proof of \eqref{DPAlem:AverIneq}, where the details are delayed to Section \ref{subsec:Lemma4_proof}. 

\begin{enumerate}
\item[(a)] Denote $\ybf(t) \triangleq \xbf(t) - \xbar(t)\1$. By \eqref{analysis:xiUpdate} and \eqref{analysis:xbarUpdate} the update of $\dot{\ybf}(t)$ can be written as
\begin{align}
\dot{\ybf}(t) 
&= -\beta\ybf(t) + \beta\A\ybf(t-\tau) - \alpha(t)(\I -\frac{1}{n}\mathbf{1}\mathbf{1}^T)\nabla F(\xbf(t))-\alpha(t)(\I -\frac{1}{n}\mathbf{1}\mathbf{1}^T)\zetabf(\vbf(t)).\label{pertbAverdelay:eq1}
\end{align}
Due to the delay term $\A\ybf(t-\tau)$ in \eqref{pertbAverdelay:eq1} one would expect an accumulation of this term for the solution $\ybf(t)$ of \eqref{pertbAverdelay:eq1}. Indeed, $\ybf(t)$ is given as
\begin{align}
\ybf(t) &= e^{-\beta t}\ybf(0) + \beta \int_{0}^{t}e^{-\beta(t-u)}\A \ybf(u-\tau)du\nonumber\\ 
&\quad - \int_{0}^{t}e^{-\beta(t-u)}\alpha(u)(\I-\frac{1}{n}\mathbf{1}\mathbf{1}^T)\left(\nabla F(\xbf(u))+\zeta(\vbf(u))\right)du.\nonumber
\end{align}
\item[(b)] To show \eqref{DPAlem:AverIneq}, we take the $2-$norm of the preceding relation and use the triangle inequality to obtain 
\begin{align}
\|\ybf(t)\|_2 \leq &e^{-\beta t}\|\ybf(0)\|_2 + \beta \int_{0}^{t}e^{-\beta(t-u)}\|\A \ybf(u-\tau)\|_2du\nonumber\\
& + \int_{0}^{t}e^{-\beta(t-u)}\left\| \alpha(u)(\I-\frac{1}{n}\mathbf{1}\mathbf{1}^T)\left(\nabla F(\xbf(u))+\zeta(\vbf(u))\right)\right\|_2du.\nonumber
\end{align}
By the Cauchy-Schwartz inequality one can show that 
\begin{align*}
\left \| \alpha(u)(\I-\frac{1}{n}\mathbf{1}\mathbf{1}^T)\nabla F(\xbf(t))\right\|_2 \leq \alpha(u)C. 
\end{align*}
Furthermore, from \eqref{analysis:xiUpdate} one can obtain 
\begin{align*}
\left \| \alpha(u)(\I-\frac{1}{n}\mathbf{1}\mathbf{1}^T)\zeta(\vbf(u))\right\|_2 \leq \alpha(u)C. 
\end{align*}
\item[(c)]  Finally, the key step of our analysis is to provide an upper bound for 
\begin{align*}
\beta \int_{0}^{t}e^{-\beta(t-u)}\|\A \ybf(u-\tau)\|_2du,
\end{align*}
\hspace{-0.2cm}which is done by applying the $Gr\ddot{o}nwall$-$Bellman$ Inequality \cite{KHALIL02}. 
\end{enumerate} 
\end{proof}

We are now ready to state our main result of this section, which is the convergence rate of  Algorithm \ref{alg:delay} to the optimal value using standard techniques in the analysis of centralized subgradient methods. One can view the update $\xbar(t)$ in \eqref{analysis:xbarUpdate} as a centralized projected subgradient used to solve problem \eqref{prob:obj}. Specifically, at any time $t\geq 0$ if each node $i\in\Vcal$ maintains a variable $z_i(t)$ to compute the time-weighted average of its estimate $x_i(t)$ and if the stepsize $\alpha(t)$ decays as $\alpha(t) = 1/\sqrt{t}$, the objective function value $F$ estimated at each $z_i(t)$ converges to the optimal value with a rate $\mathcal{O}\Big(\frac{n\tau^3\ln(t)}{(1-\gamma)^2\sqrt{t}}\Big)$, where $\gamma =\sigma_2 e^{\beta\tau} \in (0,1)$ and $\beta \in (0,\frac{\ln(1/\sigma_2)}{\tau})$. We also note that this condition on the stepsizes is also used to study the convergence rate of centralized subgradient methods \cite{Nesterov04}. 
The following Theorem is used to show the convergence rate of Algorithm \ref{alg:delay}, and its proof is given in Section \ref{subsec:DelayProofs_rate}

\begin{theorem}\label{DCRthm:ConvRate}
Suppose Assumptions \ref{assump:convexity}--\ref{assump:doublystochastic} hold. Let the trajectories of $x_i(t)$ be updated by Algorithm \ref{alg:delay}. Let $\beta \in (0,\frac{\ln(1/\sigma_2)}{\tau})$ and $\gamma =\sigma_2 e^{\beta\tau} \in (0,1)$. Let $\{\alpha(t)\}$ be a given positive scalar sequence such that $\alpha(t) = 1/\sqrt{t}$ for $t\geq 1$ and $\alpha(t) = 1$ for $t\leq 1$. Then for all $i=1,\ldots,n$,\vspace{-0.3cm}

\begin{align}
&F\left(\frac{\int_{0}^t\alpha(u)x_i(u)du}{\int_{0}^t\alpha(u)du}\mathbf{1}\right) -f^*\leq \frac{2\Gamma_0(t)+nV(\xbar(0))}{2(\sqrt{t}-1)},
\label{DCRthm:funcrate}
\end{align}
where,
\begin{align}
\Gamma_0(t) \triangleq  &  \frac{24C\left(\|\xbf(0)\|_2+2C\right)e^{\beta\tau/2}}{\beta^3(1-\gamma)^2} +  \frac{48C^2(1+\tau)}{\beta^2\gamma(1-\gamma)}+C^2\ln(t)+ \frac{48C^2\ln(\gamma t-4\tau)}{\beta^2\gamma(1-\gamma)}\cdot
\end{align}
\end{theorem}

\begin{proof}[Sketch of Proof]
As mentioned previously, the main idea of this proof is to introduce a candidate Lyapunov functional, which takes into account the impact of delays. In particular, a quadractic Lyapunov function, i.e., $(\xbar(t)-x^*)^2$, is often used in the case of no communication delay. However, since the estimates $x_i(t)$ depends on the interval $[t-\tau,t]$ we consider an extra term to study this impact. Specifically, we consider the following candidate  Razumikhin-Krasovskii Lyapunov functional $V$ \cite{HLbook}: 
\begin{align}
V(\xbar(t)) = \frac{1}{2}(\xbar(t) - x^*)^2 + \frac{\beta}{2}\int_{t-\tau}^t (\xbar(s)-x^*)^2ds.\nonumber
\end{align}  
We then show that $V$ is suffciently decreasing by considering the following two main steps.
\begin{itemize}
\item[(a)] One can show that the derivative of $V$ satisfies 
\begin{align*}
\dot{V}(\xbar(t)) \leq \frac{2C\alpha(t)}{n}\|\xbf(t) - \xbar(t)\|_2 + \frac{C^2\alpha^2(t)}{n} - \frac{\alpha(t)}{n}(F(\xbar(t)\1)-f^*).
\end{align*}
\item[(b)] Integrating both sides of the inequality in (a) and using \eqref{DPAlem:Upperbound} we can achieve the convergence rate \eqref{DCRthm:funcrate}.
\end{itemize}
\end{proof}

\begin{remark*}
Note that the convergence rate in \eqref{DCRthm:funcrate} requires each node computing the time-weighted average of its estimate. This can be done iteratively as follows. Let every node $i$ stores a variable $z_i(t)\in\mathbb{R}$ initialized at time $t=0$ with an arbitrary $z_i(0)\in\mathbb{R}$ and for all $t>0$ updated by
\begin{align}
\dot{z}_i(t) = \frac{ \alpha(t)x_i(t)-\alpha(t) z_i(t)}{S(t)},\label{DCRthm:ziUpdate}
\end{align}
where $S(0) = 0$ and $\dot{S}(t) = \alpha(t)$ for $t>0$. Then we have
\begin{align}
&\frac{d}{dt}(S(t)z_i(t)) = \dot{S}(t)z_i(t) + S(t)\dot{z}_i(t) \stackrel{\eqref{DCRthm:ziUpdate}}{=}  \alpha(t)x_i(t)\nonumber\\
& \Rightarrow z_i(t) = \frac{\int_0^t \alpha(u)x_i(u)du}{\int_0^t \alpha(u)du}\quad \forall i\in\Vcal.\nonumber
\end{align}
\end{remark*}

%% file: simulation.tex
%!TEX root = Sigmetrics18_DBS.tex

\section{Simulations}\label{sec:simulations}
In this section, we apply the distributed gradient algorithm to study the well-known linear regression problem in statistical machine learning, which is the most popular technique for data fitting \cite{HTFbook,SSBDbook}.  The goal of this problem is to find a linear relationship between a set of variables and some real value outcome. Here, we focus on quadratic loss functions, that is, given a training set $S = \{(x_i,y_i)\in\mathbb{R}^{d}\times\mathbb{R}\}$ for $i=1,\ldots,n$, we want to learn a parameter $w$ that minimizes the following least squares problem,
\begin{align}
\min_{w\in\Xcal}\sum_{i=1}^n (x_i^Tw - y_i)^2.\label{sim:prob}
\end{align}
We assume that the data sets are distributedly stored in a network of $n$ processors, i.e., each processor $i$ knows only the pair $(x_i,y_i)$.

For the purpose of simulations, we consider the discrete-time version of Algorithm \ref{alg:delay}, i.e., Eq. \eqref{DGM:discrete-time} with communication delays $\tau$.  We simulate for the case when $\Xcal = [-5,\ 5]^{d}$ where $d=10$, i.e., $w,\ x_i\in\mathbb{R}^{10}$. We consider simulated training data sets, i.e., $(x_i,y_i)$ are generated randomly with uniform distribution between $[0,1]$. We consider the performance of the distributed gradient algorithm on different sizes of network $\Gcal$, where each network is generated as follows.
\begin{enumerate}
\item In each network, we first randomly generate the nodes' coordinates in the plane with uniform distribution.
\item Then any two nodes are connected if their distance is less than a reference number $r$, e.g, $r = 0.6$ for our simulations.
\item Finally we check whether the network is connected. If not we return to step $1$ and run the program again.
\end{enumerate}
To implement our algorithm, the communication matrix $A$ is chosen as a lazy Metropolis matrix corresponding to $\Gcal$, i.e.,
\begin{align}
\A = [a_{ij}] = \left\{\begin{array}{ll}
\frac{1}{2(\max\{|\mathcal{N}_i| , |\mathcal{N}_j|\})}, & \text{ if } (i,j) \in \mathcal{E}\\
0, &\text{ if } (i,j)\notin\mathcal{E} \text{ and } i\neq j\\
1-\sum_{j\in\mathcal{N}_i}a_{ij},& \text{ if } i = j
\end{array}\right.\nonumber
\end{align}
It is straightforward to verify that the lazy Metropolis matrix $A$ satisfies Assumption \ref{assump:doublystochastic}. In all simulations considered herein, we set the stepsize $\alpha(k) = 1/\sqrt{k}$ for $k =  1,2,\ldots$ and $\alpha(0) = 1$.

In the sequel, we will compare the performance of the discretized version of distributed gradient (DG)  with distributed dual averaging (DA) \cite{Duchi2012,Rabbat2012c} for solving problem \eqref{sim:prob} in the delay-free case as well as in the case of constant delays. For DA, we chose the same stepsize $\alpha(k) = 1/\sqrt{k}$ as used in our algorithm. Simulations show that the distributed gradient algorithm outperforms distributed dual averaging in both cases.

\subsection{Delay-free case}
In the delay-free case, i.e., $\tau = 0$, we simulate DG and DA for three different sizes of networks, namely, $n = 30$, $n=40$, and $n=50$. In each simulation, we fix the number of iterations $t = 1000$ and output the worst-case distance of the function value to the optimal value, i.e., $\max_{i} |F(z_i(t))-f^*|$, where $z_i(t) = \frac{1}{T}\sum_{t=1}^{T} x_i(t)$. The simulations are shown in Fig. \ref{fig:nodelay}.

In these simulations, the performance of the DG algorithm is always slightly better than that of the DA algorithm, but overall they seem to share the same convergence rate $\mathcal{O}(\ln(t)/\sqrt{t})$, which agrees with the analytical result in Theorem \ref{DCRthm:ConvRate} and in \cite{Duchi2012,Nedic2009}.

\begin{figure}
\vspace{-3.5cm}\centering
\includegraphics[width=\columnwidth,keepaspectratio]{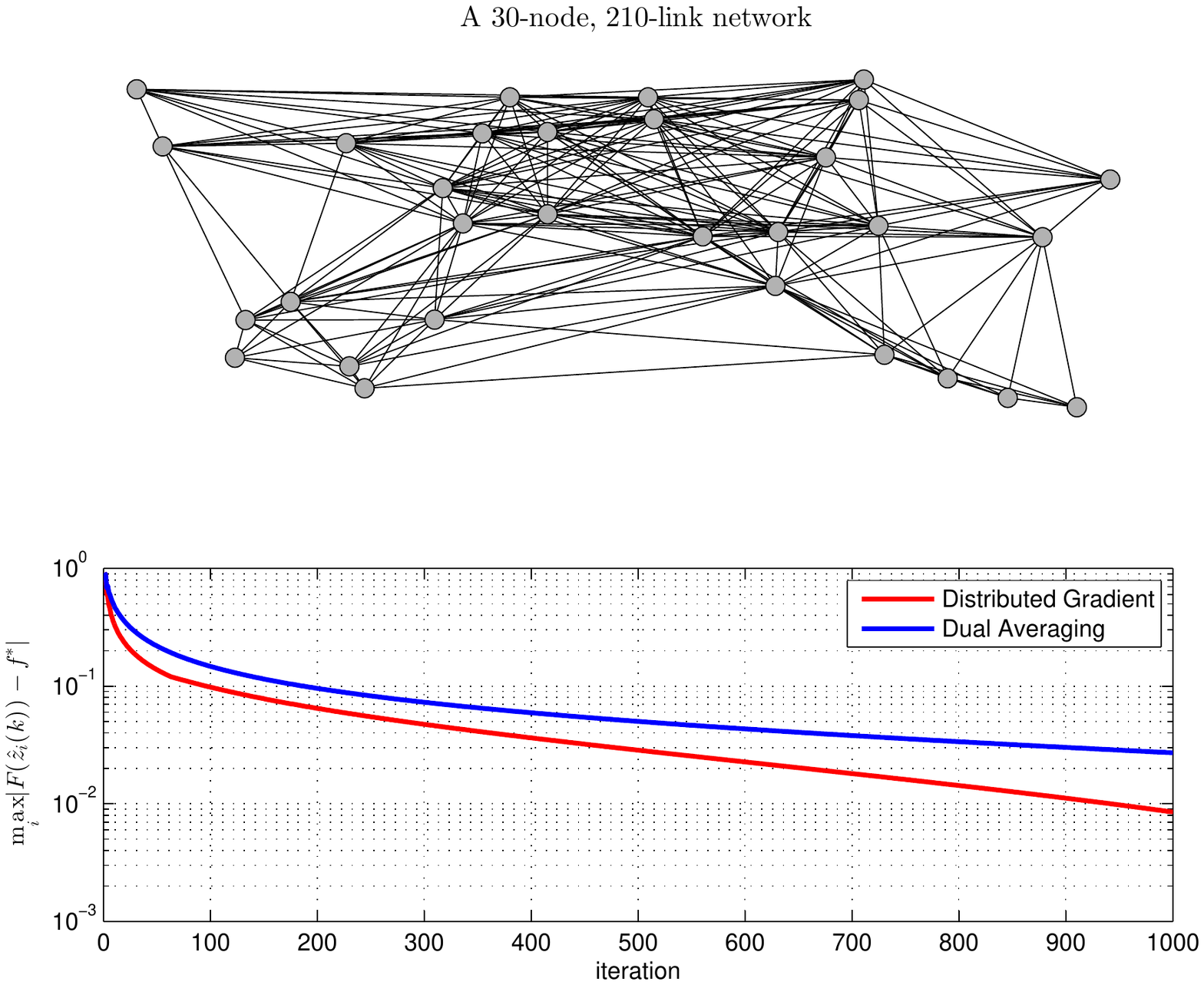}\vspace{-8cm}
\includegraphics[width=\columnwidth,keepaspectratio]{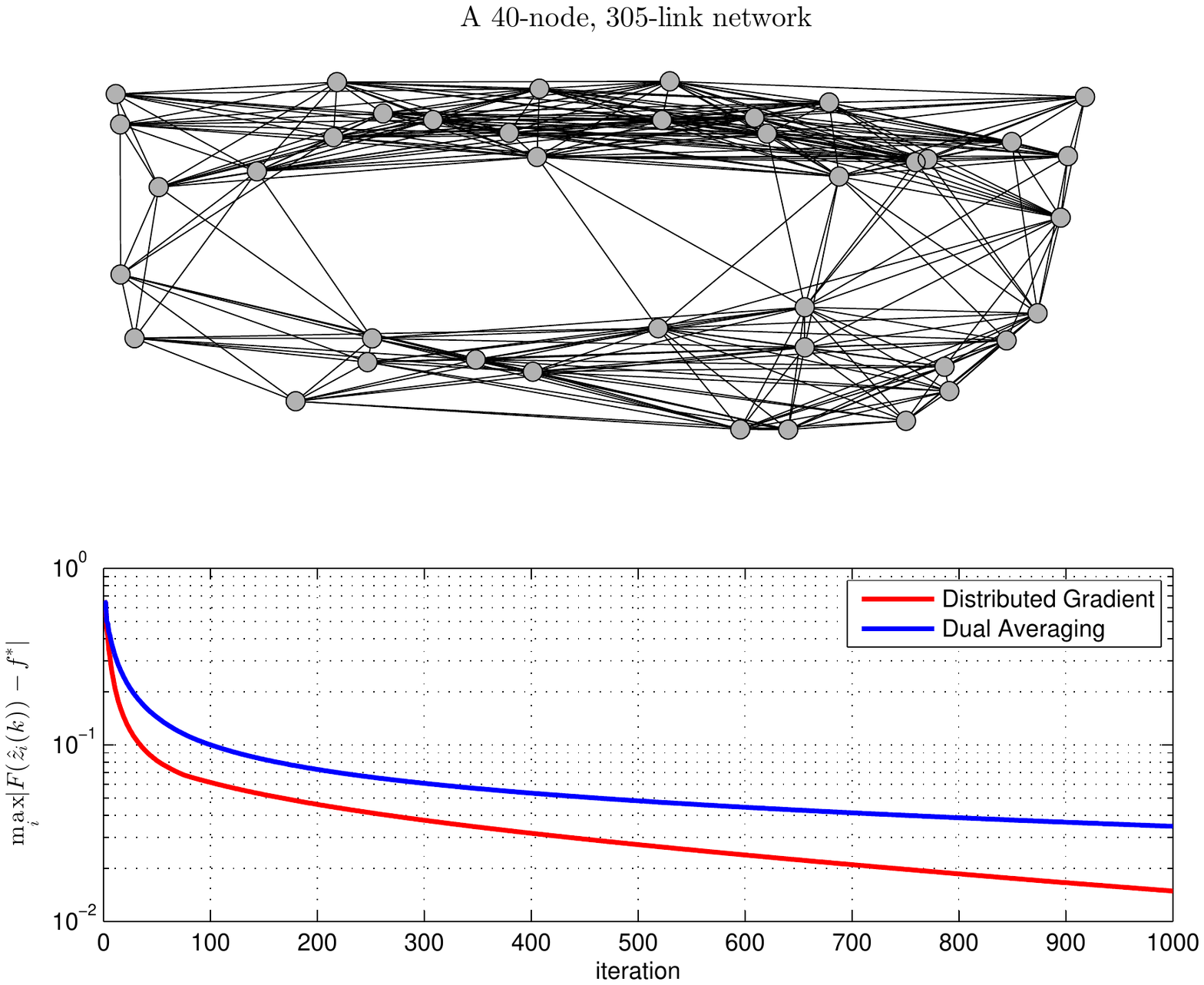}
\end{figure}
\begin{figure}
\vspace{-4cm}\centering
\includegraphics[width=\columnwidth,keepaspectratio] {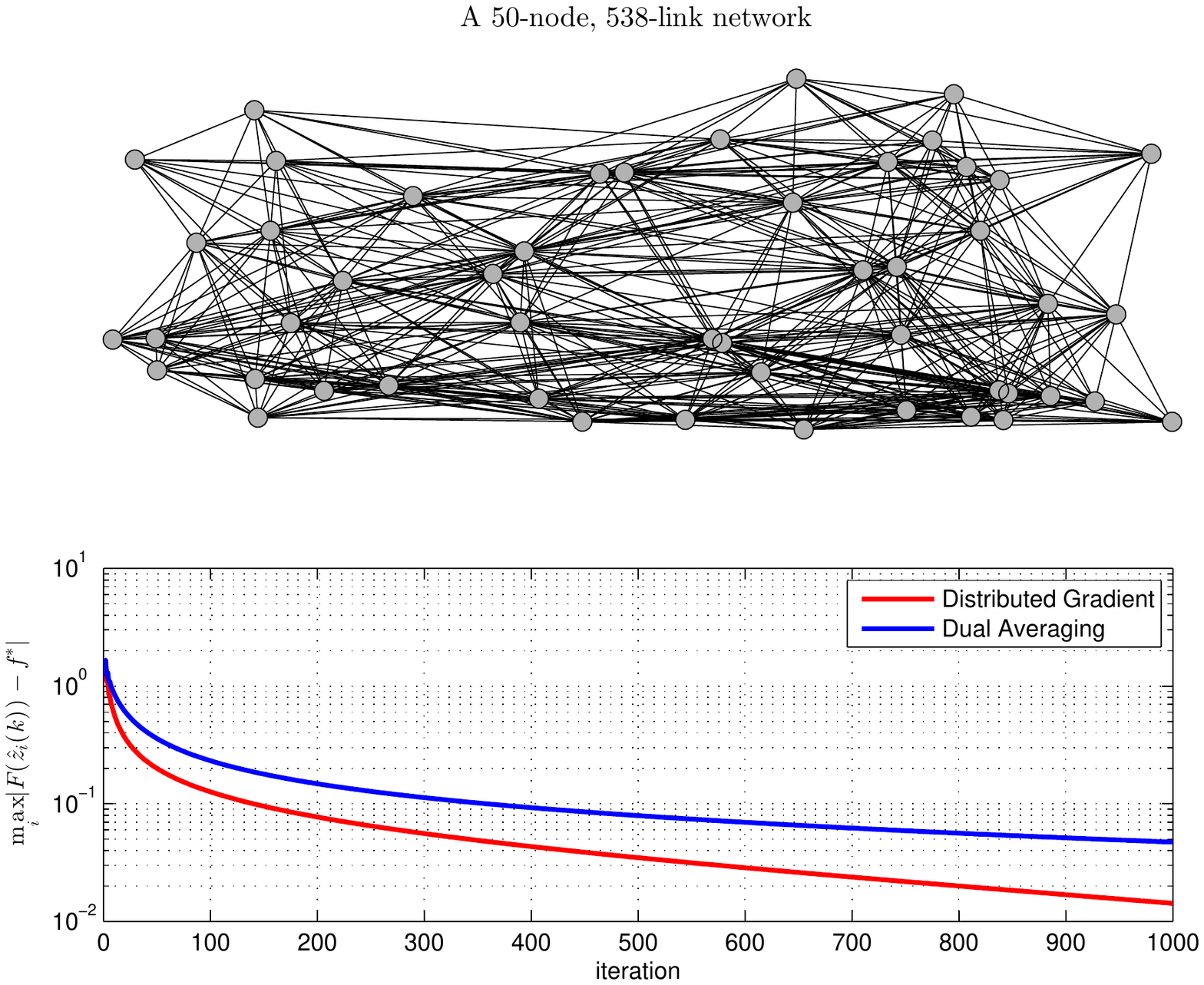}\vspace{-4.5cm}
\caption{Performance of DG and DA in delay-free networks.}\vspace{-0.5cm}
\label{fig:nodelay}
\end{figure}
\subsection{Uniform delays}
To study the impact of uniform communication delays on the performance of DG and DA, similar to the delay-free case we simulate the two algorithms for three different sizes of networks, namely, $n = 30$, $n=40$, and $n=50$. We implement DG and DA for each network, and terminate them when $\max_{i} |F(z_i(t))-f^*|\leq 0.2$. We let the delay constant $\tau$ run from $0$ to $10$ and output the number of iterations as a function on $\tau$. We plot the number of iterations as a function on the number of delay steps. The simulations are shown in Fig. \ref{fig:delay}.

We first note that the delays do influence the convergence rate of the two algorithms, that is, the greater the delay between nodes the more time the algorithms need to terminate. Second, as shown by the curve for DG the number of iterations seems to  increase as a cubic function of the number of delay steps, which agrees with our analysis in Theorem \ref{DCRthm:ConvRate}. Finally, in this example, uniform delays have a bigger impact on the performance of DA, that is, DA requires more iterations to converge than DG under the same number of delay steps.

\begin{figure}
\vspace{-4cm}\centering
\includegraphics[width=\columnwidth,,keepaspectratio]{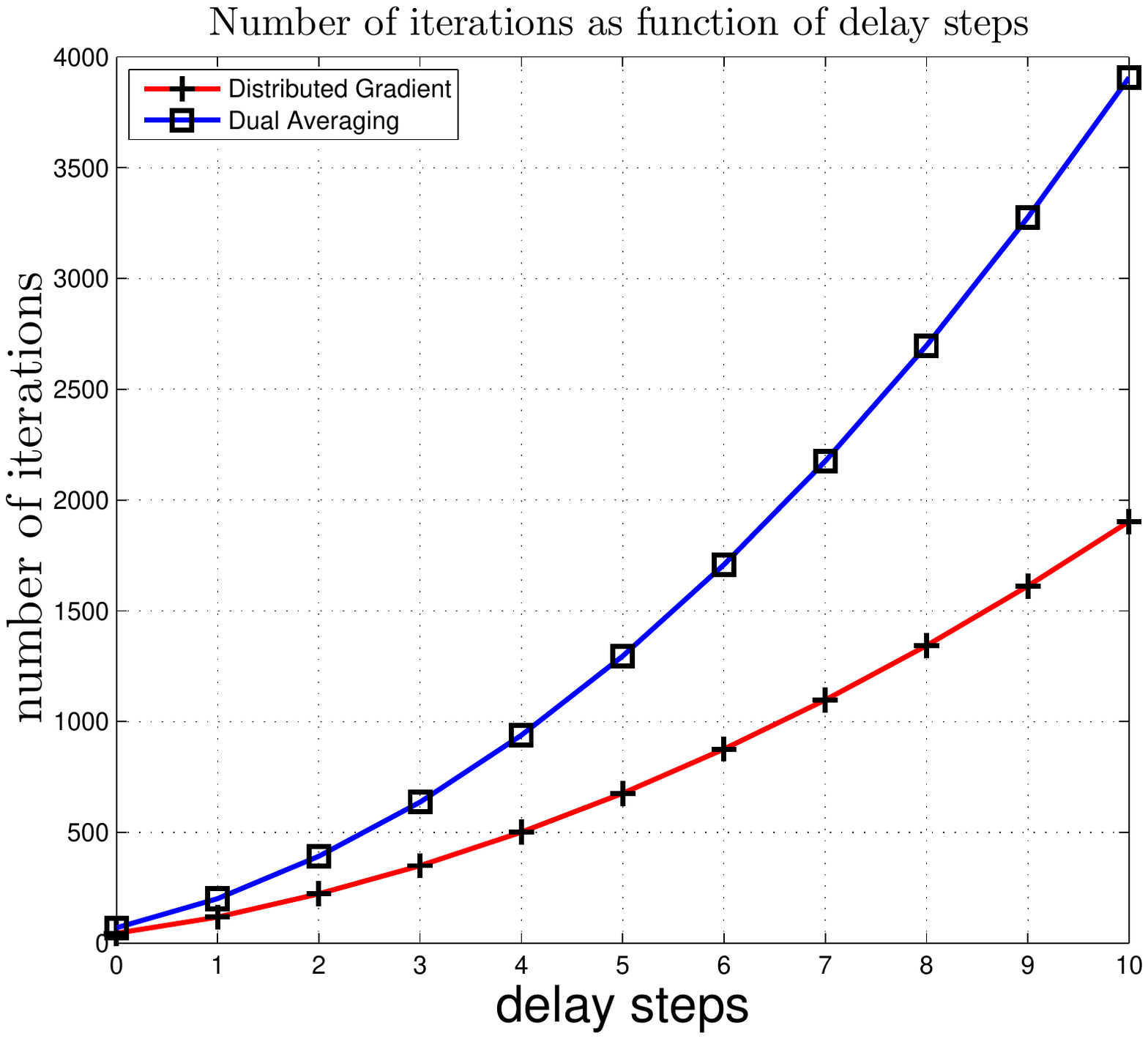}\vspace{-8cm}
\includegraphics[width=\columnwidth,,keepaspectratio]{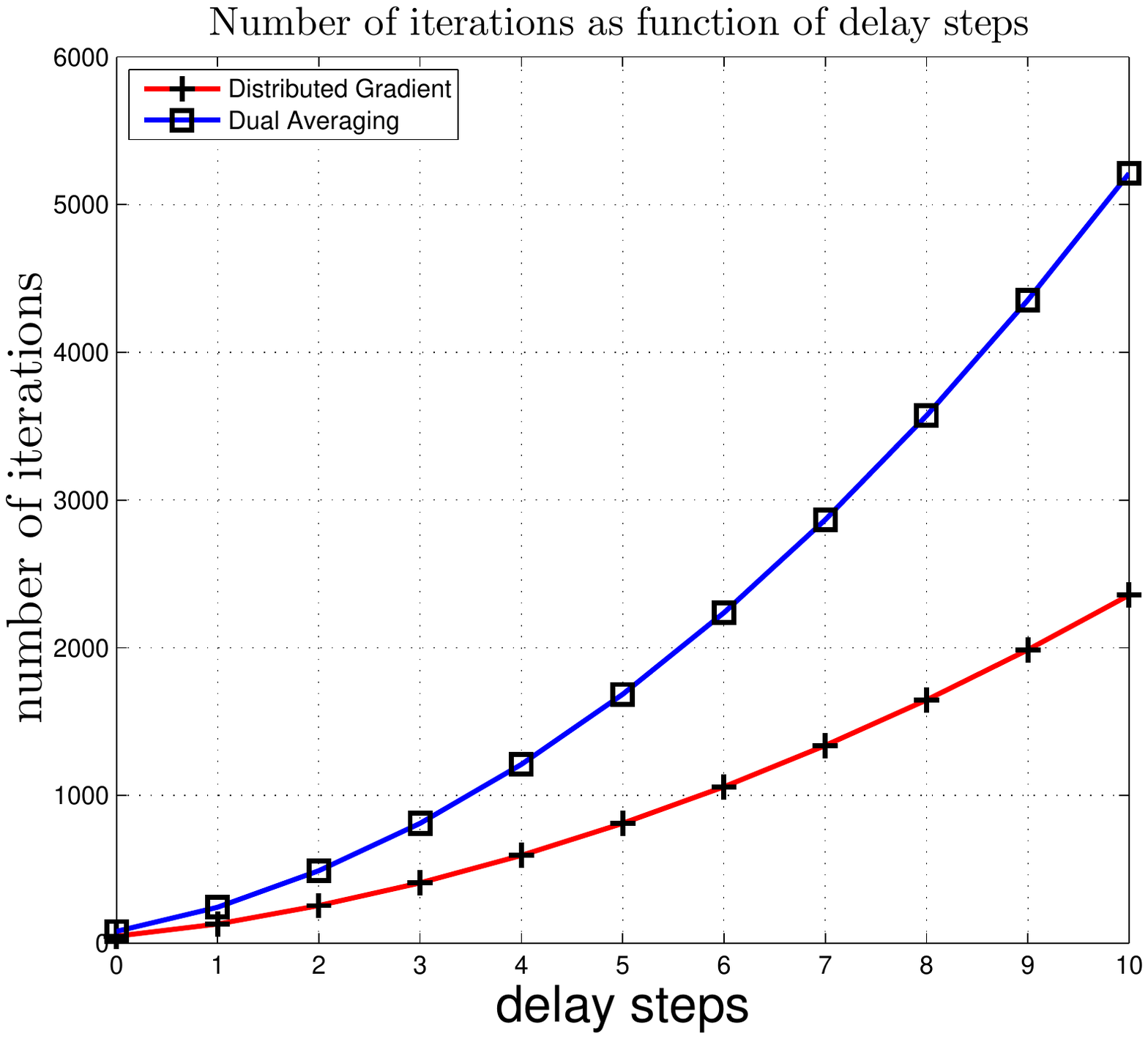}
\end{figure}
\begin{figure}
\vspace{-4cm}\centering
\includegraphics[width=\columnwidth,,keepaspectratio]{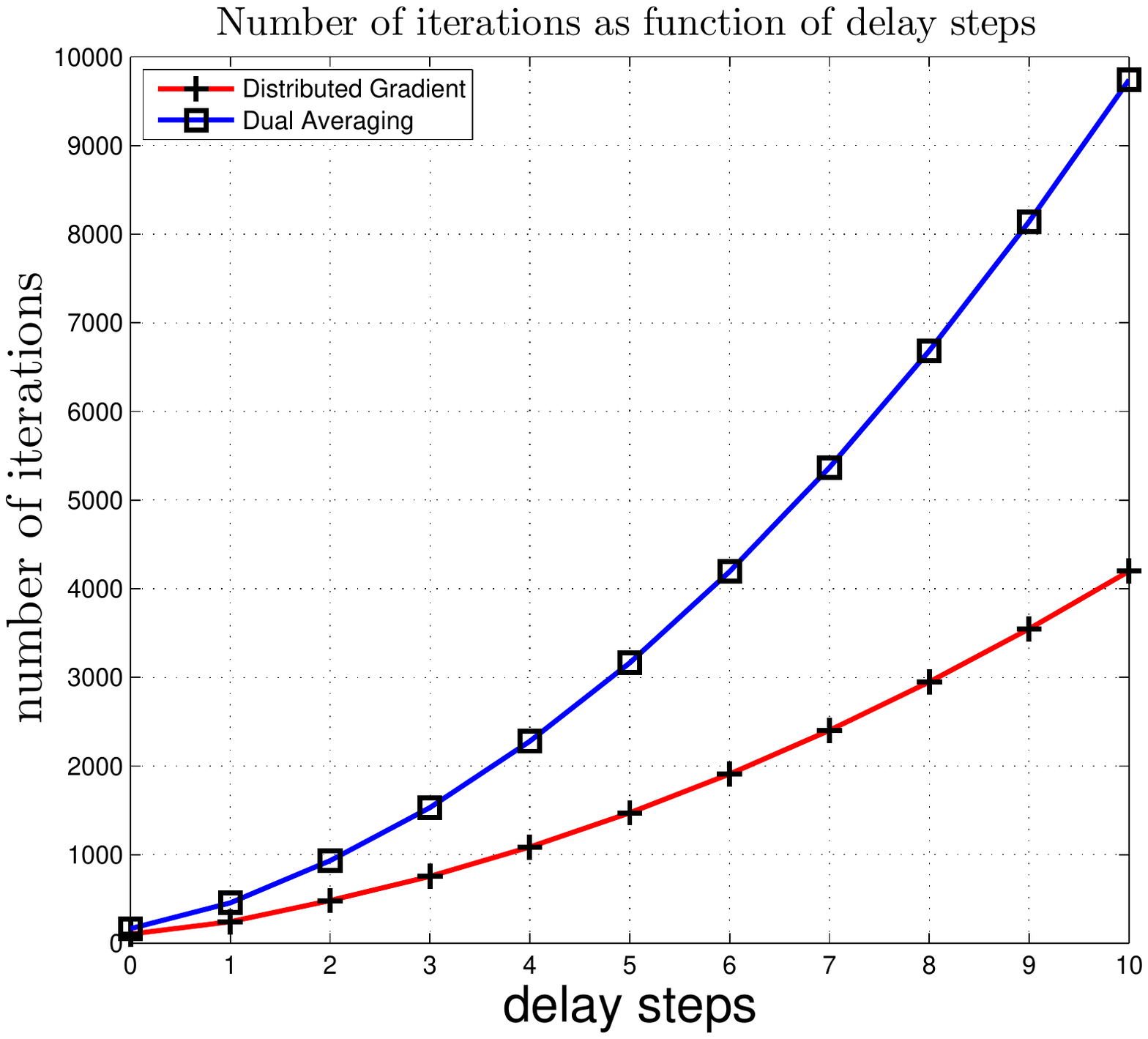}\vspace{-4.5cm}
\caption{Performance of DG and DA with delays.}\label{fig:delay}
\end{figure}

%% file: uniform_analysis.tex
%!TEX root = Sigmetrics18_DBS.tex

We provide here the complete proof of the main results presented in Section \ref{sec:delay}. In the following Lemma, we first study some important properties for the projection error $\zeta_i$ , which can be viewed as the one-dimension version of Lemma \ref{lem:Projection} for the general convex set $\Xcal$, stated in the Appendix.

\begin{lemma}\label{lem_analysis:projection}
Suppose Assumptions \ref{assump:convexity}-- \ref{assump:doublystochastic} hold. Let $v_i(t), x_i(t)$ be updated by \eqref{analysis:viUpdate} and \eqref{analysis:xiUpdate} Moreover, let $\zeta(v_i(t)) = v_i(t) - \Pcal(v_i(t))$. Then for all $i\in\Vcal$ we have
\begin{enumerate}
\item For all $t\geq0$
\begin{align}
|\zeta_i(v_i(t))| \leq |\alpha(t)f_i'(x_i(t))| \leq C_i\alpha(t). \label{lem_analysis:errorbound1}
\end{align}
\item Given any feasible direction $r_i$, i.e., 
\begin{align}
\left\{\begin{array}{cc}
r_i \leq 0 &\text{ if }\; x_i(t) = b\\
r_i \geq 0 & \text{ if }\; x_i(t) = a 
\end{array}\right.\label{lem_analysis:feas_direct}
\end{align}
We have
\begin{align}
\left(v_i(t) - r_i\right)\zeta_i(v_i(t)) \geq [\zeta_i(v_i(t))]^2.\label{lem_analysis:errorbound2}
\end{align}
\end{enumerate} 
\end{lemma}

\begin{proof}
\begin{enumerate}
\item Recall that $\zeta_i(v_i(t)) = v_i(t)-\mathcal{P}_{\Tcal_{\Xcal(x_i(t))}}$. Moreover, by \eqref{analysis:xiUpdate} we have the following three cases for all $i\in\Vcal:$ 
\begin{enumerate}
\item If $x_i(t)\in \Xcal = (a,b)$ then $\zeta_i(v_i(t)) = v_i(t) -v_i(t) = 0$.
\item If $x_i(t) = a$ then we have $0\leq \mathcal{P}_{\Tcal_{\Xcal(x_i(t))}} = v_i^{+}(t) = \max(0,v_i(t))$. If $v_i(t)\geq 0$ then $\zeta_i(v_i(t)) = 0$. Otherwise if $v_i(t)  = -\beta a + \beta\sum_{j=1}^{n}a_{ij}x_j(t-\tau) - \alpha(t)f_i'(x_i(t))<0$ then since $x_j(t-\tau)\in (a,b)$ we have $0\leq -\beta a + \beta\sum_{j=1}^{n}a_{ij}x_j(t-\tau)$. This implies that 
\begin{align*}
-\alpha(t)f_i'(x_i(t))\leq -\beta a + \beta\sum_{j=1}^{n}a_{ij}x_j(t-\tau)-\alpha(t)f_i'(x_i(t))\leq 0.
\end{align*} 
This implies that
\begin{align*}
|\zeta_i(v_i(t))| &= | v_i(t) - \mathcal{P}_{\Tcal_{\Xcal(x_i(t))}} | = |\beta\sum_{j=1}^{n}a_{ij}x_j(t-\tau) + \alpha(t)f_i'(x_i(t))| \nonumber\\ 
&\leq |\alpha(t)f_i'(x_i(t))|
\end{align*}
\item Finally, if $x_i(t) = b$ then $\mathcal{P}_{\Tcal_{\Xcal(x_i(t))}} = -v_i^{-}(t) = -\max(0,-v)\leq 0$. If $v_i(t) < 0$ then $\mathcal{P}_{\Tcal_{\Xcal(x_i(t))}} = v_i(t)$ implying $\zeta_i(v_i(t)) = 0$. Otherwise, if $v_i(t) \geq 0$ then $\mathcal{P}_{\Tcal_{\Xcal(x_i(t))}}=0$, which implies 
\begin{align*}
0&\leq -\beta x_i(t) +\beta\sum_{j=1}^{n}a_{ij}x_j(t-\tau)-\alpha(t)f_i(x_i(t))\nonumber\\ 
&= -\beta b +\beta\sum_{j=1}^{n}a_{ij}x_j(t-\tau) -\alpha(t)f_i(x_i(t))\displaybreak[0]\nonumber\\ 
&\leq \beta (b-\sum_{j=1}^{n}a_{ij}b) -\alpha(t)f_i(x_i(t))= -\alpha(t)f_i(x_i(t)).
\end{align*}  
Thus we have
\begin{align*}
|\zeta_i(v_i(t))| = |v_i(t)| &= |-\beta x_i(t) +\beta\sum_{j=1}^{n}a_{ij}x_j(t-\tau)-\alpha(t)f_i(x_i(t))|\nonumber\\
&\leq |\alpha(t)f_i(x_i(t))|
\end{align*}
\end{enumerate}
From these three cases, we have $|\zeta_i(v_i(t))|\leq |\alpha(t)f_i(x_i(t))|$, which by \eqref{prop_eq:Lipschitz} implies $|\zeta_i(v_i(t))|\leq C_i\alpha(t)$.

\item Let $r_i$ be a feasible direction, i.e., $r_i$ satisfies \eqref{lem_analysis:feas_direct}. Consider
\begin{align}
(v_i(t)-r_i)\zeta_i(v_i(t)) &= (v_i(t)-\Pcal(v_i(t))+\Pcal(v_i(t))-r_i)\zeta_i(v_i(t)) \nonumber\\ 
&= \zeta_i^2(v_i(t)) + (\Pcal(v_i(t))-r_i(t))\zeta_i(v_i(t))\nonumber\\
&= \zeta_i^2(v_i(t)) + \underbrace{(\Pcal(v_i(t))-r_i(t))(v_i(t)-\Pcal(v_i(t)))}_{q_i}\label{lem_analysis:Eq1}
\end{align}
We now investigate the second term of the previous relation for three cases
\begin{enumerate}
\item If $x_i(t)\in \Xcal = (a,b)$ then $\Pcal(v_i(t)) = v_i(t)$ implying $q_1= 0$.
\item If $x_i(t) = a$ then we have $0\leq \mathcal{P}_{\Tcal_{\Xcal(x_i(t))}} = v_i^{+}(t) = \max(0,v_i(t))$. If $v_i(t)\geq 0$ then $\Pcal(v_i(t)) = v_i(t)$ implying $q_i=0$. Otherwise if $v_i(t)<0$ then $\Pcal(v_i(t)) =0$. Since $x_i(t) = a$ we have $r_i\geq 0$, which implies $q_i\geq 0$ since $v_i(t)\leq 0$
\item Finally, if $x_i(t) = b$ then $\Pcal(v_i(t))  = -\max(0,-v)\leq 0$. If $v_i(t) < 0$ then $\Pcal(v_i(t)) = v_i(t)$ implying $q_i= 0$. Otherwise, if $v_i(t) \geq 0$ then $\Pcal(v_i(t))=0$. Since $x_i(t) = b$ we have $r_i \leq 0$, which implies $q_1\geq 0$ since $v_i(t)\geq 0$.
\end{enumerate}
Combining these three cases and by \eqref{lem_analysis:Eq1} we have \eqref{lem_analysis:errorbound2}.
\end{enumerate}
\end{proof}

\subsection{Proof of Lemma \ref{lem:pertbAverdelay}}\label{subsec:Lemma4_proof}

\begin{proof}
We start by introducing the following notation for convenience
\begin{align*}
\gbf(t)&=(I-\frac{1}{n}\mathbf{1}\mathbf{1}^T)\nabla F(\xbf(t)),\quad\quad \mathbf{h}(t) = \left(\I-\frac{1}{n}\1\1^T\right)\zetabf(\vbf(t))\quad\quad \ybf(t) = \xbf(t) - \bar{x}(t)\mathbf{1}.
\end{align*} 
\begin{itemize}
\item[(1) ] We first show the details of steps $(a)-(c)$ stated in the proof sketch of Lemma \ref{lem:pertbAverdelay}. 
\begin{itemize}
\item[(a)] By \eqref{analysis:xUpdate} and \eqref{analysis:xbarUpdate} we have,
\begin{align}
\dot{\ybf}(t) &= \dot{\xbf}(t) - \dot{\bar{x}}(t)\mathbf{1}\nonumber\\
&= -\beta\xbf(t) + \beta \A\xbf(t-\tau) + \beta\xbar(t)\1 - \beta \xbar(t-\tau)\1\nonumber\\ 
&\quad -\alpha(t)\nabla F(\xbf(t))  + \frac{\alpha(t)}{n}\1\1^T\nabla F(\xbf(t)) - \zetabf(\vbf(t)) + \frac{1}{n}\1\1^T\zetabf(\vbf(t)) \nonumber\\
&= -\beta(\xbf(t)-\xbar(t)\1) + \beta \A(\xbf(t-\tau) - \xbar(t-\tau)\1)\nonumber\\ 
&\quad  -\alpha(t)\left(\I-\frac{1}{n}\1\1^T\right)\nabla F(\xbf(t))-\left(\I-\frac{1}{n}\1\1^T\right)\zetabf(\vbf(t)) \nonumber\\
&= -\beta \ybf(t)+ \beta A \ybf(t-\tau)- \alpha(t)\gbf(t) -\mathbf{h}(t),\label{dplem:Eq1}
\end{align}
where the last equality is due to the fact that $\A$ is doubly stochastic. The solution of \eqref{dplem:Eq1} is then given as,
\begin{align}
\ybf(t)&= e^{-\beta t}\ybf(0)+ \beta \int_{0}^{t}e^{-\beta(t-u)}\A \ybf(u-\tau)du\nonumber\\ 
&\quad - \int_{0}^{t}e^{-\beta(t-u)}\left(\alpha(u)\gbf(u)+\mathbf{h}(u)\right)du.\label{dplem:Eq1a}
\end{align}
\item[(b)] Taking the $2-$norm of \eqref{dplem:Eq1a}, using the triangle inequality, and since $\|\ybf(0)\|_{2}\leq \|\xbf(0)\|_2$ we obtain
\begin{align}
\|\ybf(t)\|\leq &e^{-\beta t}\|\xbf(0)\|_2 + \int_{0}^{t}e^{-\beta(t-u)}\left(\alpha(u)\|\gbf(u)\|_2+\|\mathbf{h}(u)\|_2\right) du\nonumber\\
&+\beta\int_{0}^{t}e^{-\beta(t-u)}\|\A\ybf(u-\tau)\|_2du\label{dplem:Eq3a}
\end{align}
We first note that by the triangle inequality and \eqref{prop_eq:Lipschitz} we have
\begin{align}
\hspace{0.2cm}
\|\gbf(t)\|_2&=\left\|\left(\I-\frac{1}{n}\mathbf{1}\mathbf{1}^T\right)\nabla F(\mathbf{x}(t))\right\|_2\leq \|\nabla F(\mathbf{x}(t))\|_2=  \sqrt{\sum_{i=1}^n \left[f_i'(x_i(t))\right]^2}\nonumber\\
& \stackrel{\eqref{prop_eq:Lipschitz}}{\leq} \sqrt{\sum_{i=1}^n C_i^2}\leq C.\label{dplem:gbound}
\end{align}
Moreover, by \eqref{lem_analysis:errorbound1} we have 
\begin{align*}
\|\mathbf{h}(t)\|_2 = \left\|\left(\I-\frac{1}{n}\1\1^T\right)\zetabf(\vbf(t))\right\|_2\leq C\alpha(t).
\end{align*}
Substituting the previous relation and \eqref{dplem:gbound} into \eqref{dplem:Eq3a} we have
\begin{align}
\|\ybf(t)\|  & \leq e^{-\beta t}\|\xbf(0)\|_2 + 2C\int_{0}^{t}e^{-\beta(t-u)}\alpha(u)du+\beta\int_{0}^{t}e^{-\beta(t-u)}\|\A\ybf(u-\tau)\|_2du.\label{dplem:Eq3b1}
\end{align}
Moreover, consider the second term on the right-hand side of \eqref{dplem:Eq3b} 
\begin{align*}
\int_{0}^{t}e^{-\beta(t-u)}\alpha(u)du &= \int_{0}^{t/2}e^{-\beta(t-u)}\alpha(u)du+\int_{0}^{t/2}e^{-\beta(t-u)}\alpha(u)du\nonumber\\
&\leq \int_{0}^{t/2}e^{-\beta(t-u)}du + \alpha(t/2)\int_{0}^{t/2}e^{-\beta(t-u)}du\nonumber\\
&\leq \frac{1}{\beta}e^{-\beta t/2} + \frac{\alpha(t/2)}{\beta},
\end{align*}
where the first inequality is due to $\{\alpha(t)\}$ is non-increasing with $\alpha(0) = 1$. Substituting the previous relation into \eqref{dplem:Eq3b1} we have
\begin{align}
\|\ybf(t)\|  & \leq e^{-\beta t}\|\xbf(0)\|_2 + \frac{2C}{\beta}e^{-\beta t/2} + \frac{2C\alpha(t/2)}{\beta}+\beta\int_{0}^{t}e^{-\beta(t-u)}\|\A\ybf(u-\tau)\|_2du.\label{dplem:Eq3b}
\end{align}

\item[(c)] We now obtain an upper bound for the last term in \eqref{dplem:Eq3b}. We first recall that $\ybf(t) = (\I-\frac{1}{n}\mathbf{1}\mathbf{1}^T)\mathbf{x}(t)$ implying $\ybf(t)\notin$ span$\{\mathbf{1}\}$ since $\mathbf{1}^T\ybf(t) = 0$. Moreover since $\A$ is doubly stochastic $\A^T\A$ is also doubly stochastic, implying $\A$ has one singular value equal to $1$ and all others strictly less than $1$. Thus, by the Courant-Fisher Theorem \cite{HJbook} we have $\|\A\ybf(t)\|_2\leq \sigma_2 \|\ybf(t)\|_2$ where $\sigma_2$ is the second largest singular value of $\A$. Hence, from \eqref{dplem:Eq3b} we have
\begin{align}
\|\ybf(t)\|
&\leq e^{-\beta t}\|\xbf(0)\|_2 + \frac{2C}{\beta}e^{-\beta t/2} + \frac{2C\alpha(t/2)}{\beta} + \beta\sigma_2\int_{0}^{t}e^{-\beta(t-u)}\|\ybf(u-\tau)\|_2du\nonumber\\
&\leq \frac{\|\xbf(0)\|_2+2C}{\beta}e^{-\beta t/2} + \frac{2C\alpha(t/2)}{\beta} + \beta\sigma_2\int_{0}^{t}e^{-\beta(t-u)}\|\ybf(u-\tau)\|_2du\nonumber\\
&= \mu(t) + \beta\sigma_2 \int_{0}^{t}e^{-\beta(t-u)}\|\ybf(u-\tau)\|_2du,\label{dplem:Eq3c}
\end{align}
where $\mu(t)$ is defined as
\begin{align}
&\mu(t) =  \frac{\|\xbf(0)\|_2+2C}{\beta}e^{-\beta t/2} + \frac{2C\alpha(t/2)}{\beta}.\label{dplem:Eq3d}
\end{align}
We now apply a delayed version of the $Gr\ddot{o}nwall$-$Bellman$ Inequality for integrals to achieve an upper bound on the integral in \eqref{dplem:Eq3c}. Let $w(t)$ be a function of $t$, defined as
\begin{align}
w(t) = \int_{0}^{t}e^{\beta u}\|\ybf(u-\tau)\|_2du.\nonumber
\end{align}
By \eqref{dplem:Eq3c} we have $\|\ybf(t)\|\leq \mu(t) + \beta \sigma_2e^{-\beta t} w(t)$.  In addition, $w(t)$ is an incereasing function on $t$ with $w(0) =0$. Consider
\begin{align*}
\dot{w}(t) &= e^{\beta t} \|\ybf(t-\tau)\|_2\leq e^{\beta t}\left( \mu(t-\tau) + \beta \sigma_2 e^{-\beta(t-\tau)} w(t-\tau)\right)\nonumber\\
&= e^{\beta t} \mu(t-\tau) + \sigma_2\beta e^{\beta\tau} w(t-\tau)\leq e^{\beta t} \mu(t-\tau) + \sigma_2\beta e^{\beta\tau} w(t),
\end{align*}
where the last inequality is due to $w(t)$ is increasing, i.e., $w(t-\tau)\leq w(t)$. The preceding relation implies
\begin{align*}
\dot{w}(t) - \sigma_2\beta e^{\beta\tau} w(t) \leq e^{\beta t} \mu(t-\tau),
\end{align*}
which by multiplying both sides by $e^{-\sigma_2\beta e^{\beta\tau}t}$ we have
\begin{align*}
 \frac{d}{dt}\left(e^{-\sigma_2\beta e^{\beta\tau}t}w(t)\right) \leq e^{-\sigma_2\beta e^{\beta\tau}t}e^{\beta t} \mu(t-\tau).
\end{align*}
Taking the integeral from $0$ to $t$ on both sides of the previous equation and using $w(0) = 0$ we obtain
\begin{align}
w(t) \leq e^{\sigma_2\beta e^{\beta\tau}t}\int_{0}^{t}e^{\beta(1-\sigma_2 e^{\beta\tau})u}\mu(u-\tau) du.\label{dplem:Eq4a}
\end{align}
Thus since $\|\ybf(t)\|\leq \mu(t) + \beta \sigma_2e^{-\beta t} w(t)$ and by \eqref{dplem:Eq4a} we have
\begin{align}
\|\ybf(t)\| \leq \mu(t) + \beta\sigma_2\int_{0}^{t}e^{-\beta(1-\sigma_2 e^{\beta\tau})(t-u)}\mu(u-\tau)du,\label{dplem:Eq4b}
\end{align}
which is \eqref{DPAlem:AverIneq} since $\gamma =\sigma_2 e^{\beta\tau}$.
\end{itemize}

\item[(2).] We now show \eqref{DPAlem:AsympConv}. Since $\lim_{t\rightarrow\infty}\alpha(t) = 0$ we first have $\lim_{t\rightarrow\infty}\mu(t) = 0$ by \eqref{dplem:Eq3d}. Second, Eq. \eqref{dplem:Eq4b} can be written as
\begin{align}
\int_{0}^{t}e^{-\beta(1-\gamma)(t-u)}\mu(u-\tau) du = &\frac{\|\xbf(0)\|_2+2C}{\beta}\int_{0}^{t}e^{-\beta(1-\gamma)(t-u)}e^{-\beta(u-\tau)/2}du\nonumber\\
&+\frac{2C}{\beta}\int_{u=0}^{t}e^{-\beta(1-\gamma)(t-u)}\alpha((u-\tau)/2)du.\label{dplem:Eq5}
\end{align}
On the one hand, taking the limit as $t\rightarrow\infty$ on the first term on the right hand side of \eqref{dplem:Eq5} gives,
\begin{align}
\lim_{t\rightarrow\infty} \int_{0}^{t}e^{-\beta(1-\gamma)(t-u)}e^{-\beta(u-\tau)/2}du &=\lim_{t\rightarrow\infty} e^{-\beta(1-\gamma)t+\beta\tau/2}\int_{u=0}^{t}e^{\beta(1/2-\gamma) u}du\nonumber\\
&= e^{\beta\tau/2}\lim_{t\rightarrow\infty} e^{-\beta(1-\gamma)t}\frac{e^{\beta(1/2-\gamma) t}-1}{\beta(1/2-\gamma)}= 0. \label{dplem:Eq5a}
\end{align}
On the other hand, consider the second term in  \eqref{dplem:Eq5},
\begin{align}
&\lim_{t\rightarrow\infty}\int_{u=0}^{t}e^{-\beta(1-\gamma)(t-u)}\alpha((u-\tau)/2)du\nonumber\\
&= \lim_{t\rightarrow\infty}\int_{u=0}^{t/2}e^{-\beta(1-\gamma)(t-u)}\alpha((u-\tau)/2)du + \lim_{t\rightarrow\infty}\int_{u=t/2}^{t}e^{-\beta(1-\gamma)(t-u)}\alpha((u-\tau)/2)du\nonumber\\
&\leq \lim_{t\rightarrow\infty}\int_{u=0}^{t/2}e^{-\beta(1-\gamma)(t-u)}du + \lim_{t\rightarrow\infty}\alpha((u-2\tau)/4)\int_{u=t/2}^{t}e^{-\beta(1-\gamma)(t-u)}du\nonumber\\
&\leq \lim_{t\rightarrow\infty}\frac{e^{-\beta(1-\gamma)t/2}}{\beta(1-\gamma)} + \lim_{t\rightarrow\infty}\frac{\alpha((u-2\tau)/4)}{\beta(1-\gamma)} = 0,\label{dplem:Eq5b}
\end{align}
where the last equality is due to $\gamma\in (0,1)$ and $\lim_{t\rightarrow\infty}\alpha(t)= 0.$ Using the preceding relation and \eqref{dplem:Eq5a} into  \eqref{dplem:Eq5} we have
\begin{align}
\lim_{t\rightarrow\infty}\int_{0}^{t}e^{-\beta(1-\gamma)(t-u)}\mu(u-\tau) = 0,
\end{align}
which together with $\lim_{t\rightarrow\infty} \mu(t)=0$ and by \eqref{DPAlem:AverIneq} give \eqref{DPAlem:AsympConv}.
\item[(3)]
Recall from \eqref{DPAlem:AverIneq} that
\begin{align}
\int_{0}^{t}\alpha(u)\|\ybf(u)\|_2 du &\leq \int_{0}^{t}\alpha(u)\mu(u)du+ \int_{u=0}^{t}\alpha(u)\int_{s=0}^{u}e^{-\beta(1-\gamma)(u-s)}\mu(s-\tau) ds du\label{dplem:Eq6}.
\end{align}
where 
\begin{align*}
&\mu(t) =  \frac{\|\xbf(0)\|_2+2C}{\beta}e^{-\beta t/2} + \frac{2C\alpha(t/2)}{\beta}.
\end{align*}
We first analyze the first-term on the right-hand side of \eqref{dplem:Eq6}. Si
\begin{align}
\int_{0}^{t}\alpha(u)\mu(u)du&\leq \frac{\|\xbf(0)\|_2+2C}{\beta}\int_{0}^{t}\alpha(u)e^{-\beta u/2}du + \int_{0}^{t}\alpha(u)\frac{2C\alpha(t/2)}{\beta}du\nonumber\\
&\leq  \frac{\|\xbf(0)\|_2+2C}{\beta}\int_{0}^{t}e^{-\beta u/2}du + \frac{2C}{\beta}\int_{0}^{t}\alpha^2(u/2)du\nonumber\\
&\leq  \frac{2\|\xbf(0)\|_2+4C}{\beta^2}+ \frac{2C}{\beta}\int_{0}^{t}\alpha^2(u/2)du,\label{dplem:Eq7}
\end{align}
where the second inequality is due to $\alpha(t)$ is non-increasing, positive, and $\alpha(0) = 1$. 
Second, we now consider the second term on the right-hand side of \eqref{dplem:Eq6}. We first have
\begin{align}
&\int_{u=0}^{t}\alpha(u)\int_{s=0}^{u}e^{-\beta(1-\gamma)(u-s)}e^{-\beta(s-\tau)/2}dsdu\nonumber\\
&\leq e^{\beta\tau/2}\int_{u=0}^{t}\int_{s=0}^{u}e^{-\beta(1-\gamma)u} e^{\beta(1-\gamma)s/2}dsdu\displaybreak[0]\nonumber\\ 
&\leq \frac{2e^{\beta\tau/2}}{\beta(1-\gamma)}\int_{0}^{t}e^{-\beta(1-\gamma)u/2}du\leq \frac{4e^{\beta\tau/2}}{\beta^2(1-\gamma)^2}.\label{dplem:Eq8a}
\end{align}
We now consider 
\begin{align}
&\int_{u=0}^{t}\alpha(u)\int_{s=0}^{u}e^{-\beta(1-\gamma)(u-s)}\alpha((t-\tau)/2)dsdu\leq \int_{u=0}^{t}\int_{s=0}^{u}e^{-\beta(1-\gamma)(u-s)}\alpha^2((s-\tau)/2)dsdu\nonumber\\
&= \int_{u=0}^{t}e^{-\beta(1-\gamma)u}\left(\int_{s=0}^{u/2}e^{\beta(1-\gamma)s}\alpha^2((s-\tau)/2)ds+\int_{s=u/2}^{u}e^{\beta(1-\gamma)s}\alpha^2((s-\tau)/2)ds\right)du \nonumber\\
&\leq \int_{u=0}^{t}e^{-\beta(1-\gamma)u}\left(\int_{s=0}^{u/2}e^{\beta(1-\gamma)s}ds+\alpha^2((s-2\tau)/4)\int_{s=u/2}^{u}e^{\beta(1-\gamma)s}ds\right)du \nonumber\\
&\leq \frac{1}{\beta(1-\gamma)}\int_{u=0}^{t}e^{-\beta(1-\gamma)u/2}+\alpha^2((s-2\tau)/4) du \nonumber\\
&\leq \frac{2}{\beta^2(1-\gamma)^2} + \frac{1}{\beta(1-\gamma)}\int_{u=0}^{t}\alpha^2((s-2\tau)/4) du \nonumber\\
\label{dplem:Eq8b}
\end{align}
Substituting \eqref{dplem:Eq8a} into \eqref{dplem:Eq8b} into the second term on the right-hand side of \eqref{dplem:Eq6} we obtain
\begin{align}
&\int_{u=0}^{t}\alpha(u)\int_{s=0}^{u}e^{-\beta(1-\gamma)(u-s)}\mu(s-\tau) ds du\nonumber\\
&\leq\frac{4\left(\|\xbf(0)\|_2+2C\right)e^{\beta\tau/2}}{\beta^3(1-\gamma)^2}+\frac{4C}{\beta^3(1-\gamma)^2}+\frac{2C}{\beta^2(1-\gamma)}\int_{0}^{t}\alpha^2(\gamma u/4-\tau)du\nonumber\\
&\leq\frac{4\left(\|\xbf(0)\|_2+3C\right)e^{\beta\tau/2}}{\beta^3(1-\gamma)^2}+\frac{2C}{\beta^2(1-\gamma)}\int_{0}^{t}\alpha^2(\gamma u/4-\tau)du. \label{dplem:Eq9}
\end{align}
By adding \eqref{dplem:Eq9} to \eqref{dplem:Eq7} we obtain from \eqref{dplem:Eq6} that
\begin{align}
&\int_{0}^t \alpha(u)\|\ybf(u)\|_2du\nonumber\\
&\leq \frac{2\|\xbf(0)\|_2+4C}{\beta^2}+ \frac{2C}{\beta}\int_{0}^{t}\alpha^2(u/2)du\nonumber\\ 
&\quad +\frac{4\left(\|\xbf(0)\|_2+3C\right)e^{\beta\tau/2}}{\beta^3(1-\gamma)^2}+\frac{2C}{\beta^2(1-\gamma)}\int_{0}^{t}\alpha^2(\gamma u/4-\tau)du\nonumber\\
&\leq \frac{8\left(\|\xbf(0)\|_2+2C\right)e^{\beta\tau/2}}{\beta^3(1-\gamma)^2}+ \frac{4C}{\beta^2(1-\gamma)}\int_{0}^{t}\alpha^2(\gamma u/4-\tau)du,\label{dplem:Eq10}
\end{align}
where the last inequality is due to $\gamma\in(0,1)$ and $\alpha(t)$ is non-increasing positive sequence, i.e., $\alpha^2(u/2)\leq \alpha^2(\gamma u/4-\tau)$ for $\tau > 0$. This shows  \eqref{DPAlem:Upperbound}. 
\end{itemize}
\end{proof}

\subsection{Proof Theorem \ref{DCRthm:ConvRate}}\label{subsec:DelayProofs_rate}

\begin{proof}
Let $x^*$ be a solution of problem \eqref{prob:obj}. Consider a candidate Razumikhin-Krasovskii Lyapunov functional $V$ \cite{HLbook} defined as
\begin{align}
V(\bar{x}(t)) = \frac{1}{2}(\bar{x}(t)-x^*)^2 + \frac{\beta}{2} \int_{t-\tau}^t (\bar{x}(s)-x^*)^2ds,\quad t\geq0\label{DCRthm:Lyapunov},
\end{align}
whose derivative is given as
\begin{align}
&\dot{V}(\bar{x}(t))\nonumber\\
&= (\bar{x}(t)-x^*)\dot{\bar{x}}+ \frac{\beta}{2}\Big[(\bar{x}(t)-x^*)^2-(\bar{x}(t-\tau)-x^*)^2\Big]\nonumber\\
&= (\bar{x}(t)-x^*)\Big( - \beta\bar{x}(t)+\beta\bar{x}(t-\tau)-\frac{\alpha(t)}{n}\sum_{i=1}^n f_i'(x_i(t)) - \bar{\zeta}(t)\Big)+\frac{\beta(\bar{x}(t)-x^*)^2-\beta(\bar{x}(t-\tau)-x^*)^2}{2}\nonumber\\
&= \underbrace{-\frac{\alpha(t)}{n}\sum_{i=1}^n (\bar{x}(t)-x^*) f_i'(x_i(t))}_{W_1} \underbrace{-\frac{1}{n}\sum_{i=1}^n (\bar{x}(t)-x^*) z_i(x_i(t))}_{W_2}\nonumber\\ 
&\quad +\beta(\bar{x}(t)-x^*)(\bar{x}(t-\tau) - \bar{x}(t)) + \frac{\beta(\bar{x}(t)-x^*)^2-\beta(\bar{x}(t-\tau)-x^*)^2}{2}\nonumber\\
&= W_1+W_2 -\frac{\beta(\bar{x}(t)-\bar{x}(t-\tau))^2}{2}\nonumber\\ 
&\quad + \frac{\beta(\bar{x}(t-\tau)-x^*)^2-\beta(\bar{x}(t)-x^*)^2}{2}+ \frac{\beta(\bar{x}(t)-x^*)^2-\beta(\bar{x}(t-\tau)-x^*)^2}{2}\nonumber\\
&= W_1+ W_2 -\frac{\beta}{2}(\bar{x}(t)-\bar{x}(t-\tau))^2\leq W_1+W_2.\label{DCRthm:Eq1}
\end{align}
We first have
\begin{align}
W_1 &= -\frac{\alpha(t)}{n}\sum_{i=1}^n(\bar{x}(t)-x_i(t) + x_i(t) - x^*)f_i'(x_i(t))\nonumber\\
&= -\frac{\alpha(t)}{n}\sum_{i=1}^n(\bar{x}(t)-x_i(t))f_i'(x_i(t)) -\frac{\alpha(t)}{n} \sum_{i=1}^n(x_i(t) - x^*)f_i'(x_i(t))\nonumber\\
&\leq \frac{\alpha(t)}{n}\sum_{i=1}^n|\bar{x}(t)-x_i(t)|\;|f_i'(x_i(t))| -\frac{\alpha(t)}{n}(F(\xbf(t))-f^*)\nonumber\\
&\leq \frac{\alpha(t)C}{n}\|\xbf(t)-\bar{x}(t)\mathbf{1}\|_{2} - \frac{\alpha(t)}{n}(F(\xbf(t))-f^*)\nonumber\\
&= \frac{\alpha(t)C}{n}\|\xbf(t)-\bar{x}(t)\mathbf{1}\|_{2} - \frac{\alpha(t)}{n}(F(\xbf(t))-F(\bar{x}(t)\mathbf{1}))-\frac{\alpha(t)}{n}(F(\bar{x}(t)\mathbf{1}) -f^*)\nonumber\\
&\leq \frac{2\alpha(t)C}{n}\|\xbf(t)-\bar{x}(t)\mathbf{1}\|_2 - \frac{\alpha(t)}{n}(F(\bar{x}(t)\mathbf{1})-f^*).\label{DCRthm:Eq1c}
\end{align}
Second, let $r_i(t)$ be defined as
\begin{align*}
r_i(t) = x^* - x_i(t) - \beta x_i(t) + \beta\sum_{j=1}^n a_{ij}x_j(t-\tau),
\end{align*}
and recall from \eqref{lem_analysis:feas_direct} that $r_i(t)$ is a feasible direction if 
\begin{align}
\left\{\begin{array}{cc}
r_i \leq 0 &\text{ if }\; x_i(t) = B\\
r_i \geq 0 & \text{ if }\; x_i(t) = 0 
\end{array}\right.\label{DCRthm:Eq1d}
\end{align}
Indeed, if $x_i(t) = 0$ then $r_i(t)\geq$ since $x^*,x_j(t-\tau)\in(0,B)$ $\forall j\in\Vcal$ and $\A$ is doubly stochastic. On the other hand, if $x_i(t) = B$ then $r_i(t) = x^* +\beta\sum_{j=1}^na_{ij}x_j(t-\tau) - (1+\beta)B\leq 0$. Thus, we have $r_i(t)$ is a feasible direction, i.e., $r_i(t)$ satisfies \eqref{DCRthm:Eq1d}. We now consider the term $W_2$
\begin{align}
&W_2= -\frac{1}{n}\sum_{i=1}^n (\bar{x}(t)-x^*) \zeta_i(t))\nonumber\\
&= -\frac{1}{n} \sum_{i=1}^n\left(\xbar(t)-(1+\beta)x_i(t) + \beta\sum_{j=1}^n a_{ij}x_j(t-\tau) - v_i(t)\right)\zeta_i(t)\nonumber\\ 
&\quad - \frac{1}{n} \sum_{i=1}^n\left(v_i(t) + (1+\beta)x_i(t) - \beta\sum_{j=1}^n a_{ij}x_j(t-\tau)-x^*\right)\zeta_i(t)\nonumber\\
&= -\frac{1}{n} \sum_{i=1}^n\left(\xbar(t)-(1+\beta)x_i(t) + \beta\sum_{j=1}^n a_{ij}x_j(t-\tau) - v_i(t)\right)\zeta_i(t)\nonumber\\ 
&\quad - \frac{1}{n} \sum_{i=1}^n\left(v_i(t)-r_i(t)\right)\zeta_i(t), \label{DCRthm:Eq2}
\end{align}
where by \eqref{analysis:viUpdate} the first sum is equivalent to 
\begin{align*}
&-\frac{1}{n} \sum_{i=1}^n\left(\xbar(t)-(1+\beta)x_i(t) + \beta\sum_{j=1}^n a_{ij}x_j(t-\tau) - v_i(t)\right)\zeta_i(t)\nonumber\\
& = -\frac{1}{n} \sum_{i=1}^n\left(\xbar(t)-x_i(t)+\alpha(t)f_i'(x_i(t))\right)\zeta_i(t)\displaybreak[0]\nonumber\\
&\leq \frac{1}{n}\sum_{i=1}^n|\xbar(t) - x_i(t)||\zeta_i(t)| + \frac{1}{n} \sum_{i=1}^n|\alpha(t)f_i'(x_i(t))||\zeta_i(t)|\nonumber\\
&\overset{\eqref{lem_analysis:errorbound1}}{\leq} \frac{C \alpha(t)}{n}\|\xbf(t) - \xbar(t)\mathbf{1}\|_2 + \frac{C^2\alpha^2(t)}{n}.
\end{align*}
In addition, since $r_i(t)$ is a feasible direction, by \eqref{lem_analysis:errorbound2} the second sum in \eqref{DCRthm:Eq2} is upper bounded by  
\begin{align*}
 &- \frac{1}{n} \sum_{i=1}^n\left(v_i(t)-r_i(t)\right)\zeta_i(t) \leq  - \frac{1}{n} \sum_{i=1}^n\zeta_i^2(t) = -\frac{1}{n}\|\zetabf(t)\|_2^2.
\end{align*}
Applying the preceding two relations into \eqref{DCRthm:Eq2} we obtain
\begin{align}
W_2 \leq \frac{C \alpha(t)}{n}\|\xbf(t) - \xbar(t)\mathbf{1}\|_{2} + \frac{C^2\alpha^2(t)}{n} - \frac{1}{n}\|\zetabf(t)\|_2^2\leq \frac{C \alpha(t)}{n}\|\xbf(t) - \xbar(t)\mathbf{1}\|_{2} + \frac{C^2\alpha^2(t)}{n}. \label{DCRthm:Eq3}
\end{align}
Thus, substituting \eqref{DCRthm:Eq1c} amd \eqref{DCRthm:Eq3} into \eqref{DCRthm:Eq1} we obtain 
\begin{align}
\dot{V}(\bar{x}(t))&\leq \frac{3\alpha(t)C}{n}\|\xbf(t)-\bar{x}(t)\mathbf{1}\|_2  +\frac{C^2\alpha^2(t)}{n}- \frac{\alpha(t)}{n}(F(\bar{x}(t)\mathbf{1})-f^*) .\label{DCRthm:Eq4} 
\end{align}
By \eqref{DPAlem:Upperbound} in Lemma \ref{lem:pertbAverdelay} we have
\begin{align}
\int_{0}^t \alpha(u)\|\ybf(u)\|_2du\leq \frac{8\left(\|\xbf(0)\|_2+2C\right)e^{\beta\tau/2}}{\beta^3(1-\gamma)^2}+ \frac{4C}{\beta^2(1-\gamma)}\int_{0}^{t}\alpha^2(\gamma u/4-\tau)du\cdot
\label{DCRthm:Eq5}
\end{align}
Under the assumptions on $\alpha(t)$, i.e., $\alpha(t) = 1$ for $t\leq 1$ and $\alpha(t) = 1/\sqrt{t}$ for $t\geq1$, consider the following
\begin{align}
\int_{0}^{t}\alpha^2(\gamma u/4-\tau)du &= \frac{4}{\gamma}\int_{-\tau}^{\frac{\gamma t}{4}-\tau}\alpha^2(u)du = \frac{4}{\gamma}\int_{-\tau}^{1}\alpha^2(u)du+\frac{4}{\gamma}\int_{1}^{\frac{\gamma t}{4}-\tau}\alpha^2(u)du\nonumber\\
&= \frac{4(1+\tau)}{\gamma}+\frac{4}{\gamma}\int_{1}^{\frac{\gamma t}{4}-\tau}\frac{1}{t}du = \frac{4(1+\tau)}{\gamma}+\frac{4\ln(\frac{\gamma t}{4}-\tau)}{\gamma}\nonumber\\
&\leq \frac{4(1+\tau)}{\gamma}+\frac{4\ln(\gamma t-4\tau)}{\gamma}\cdot\label{DCRthm:Eq6}
\end{align}
Substituting \eqref{DCRthm:Eq6} into \eqref{DCRthm:Eq5} to obtain
\begin{align}
&3C\int_{0}^t \alpha(u)\|\ybf(u)\|_2du+\frac{C^2}{n}\int_{0}^t \alpha^2(u)du\nonumber\\
&\leq  \frac{24C\left(\|\xbf(0)\|_2+2C\right)e^{\beta\tau/2}}{\beta^3(1-\gamma)^2} +  \frac{48C^2(1+\tau)}{\beta^2\gamma(1-\gamma)}+C^2\ln(t)+ \frac{48C^2\ln(\gamma t-4\tau)}{\beta^2\gamma(1-\gamma)}\nonumber\\
&\triangleq \Gamma_0(t).\label{DCRthm:Eq7}
\end{align}
Taking the integral of both sides in \eqref{DCRthm:Eq1} and using \eqref{DCRthm:Eq4} we obtain
\begin{align}
V(\bar{x}(t)) - V(\bar{x}(0))&\leq \frac{3C}{n}\int_{0}^t \alpha(u)\|\ybf(u)\|_2du + \frac{C^2}{n}\int_{0}^t \alpha^2(u)du - \frac{1}{n}\int_{0}^t \alpha(u)(F(\bar{x}(u)\mathbf{1})-f^*)du\nonumber\\
&\leq \frac{\Gamma_0(t)}{n}- \frac{1}{n}\int_{0}^t \alpha(u)(F(\bar{x}(u)\mathbf{1})-f^*).\label{DCRthm:Eq8}
\end{align}
Rearranging \eqref{DCRthm:Eq8} and dropping $V(\bar{x}(t))$ gives
\begin{align}
&\int_{0}^t \alpha(u)(F(\bar{x}(u)\mathbf{1})-f^*)du\leq 2\Gamma_0(t)+ nV(\bar{x}(0)).\nonumber
\end{align}
Thus, dividing both sides of the preceding relation by $\int_{0}^t \alpha(u)du = 1 +  \int_{1}^t \frac{1}{\sqrt{u}}du \leq 2(\sqrt{t}-1)$ we obtain
\begin{align}
&\frac{\int_{0}^t \alpha(u)(F(\bar{x}(u)\mathbf{1})-f^*)du}{\int_{0}^t \alpha(u)du}\leq \frac{\Gamma_0(t)+ nV(\bar{x}(0))}{2(\sqrt{t}-1)},\nonumber
\end{align}
which by Jensen's inequality implies
\begin{align}
F\left(\!\!\!\begin{array}{c}
\frac{\int_{0}^t \alpha(u)\bar{x}(u)du}{\int_{0}^t \alpha(u)du}
\end{array}\mathbf{1}\right)-f^* &\leq \frac{\Gamma_0(t)+ nV(\bar{x}(0))}{2(\sqrt{t}-1)}\cdot\label{DCRthm:Eq9}
\end{align}
Moreover, we have
\begin{align}
F\left(\frac{\int_0^t \alpha(u)x_i(u)du}{\int_0^t \alpha(u)du}\mathbf{1}\right) - F\left(\!\!\!\begin{array}{c}
\frac{\int_{0}^t \alpha(u)\bar{x}(u)du}{\int_{0}^t \alpha(u)du}
\end{array}\!\!\!\mathbf{1}\right)\leq C\left |\begin{array}{c}
\frac{\int_{0}^t \alpha(u)(x_i(u)-\bar{x}(u))du}{\int_{0}^t \alpha(u)du}\end{array}\right|\stackrel{\eqref{DCRthm:Eq4}}{\leq}\frac{\Gamma_0(t)}{2(\sqrt{t}-1)}\cdot\label{DCRthm:Eq11}
\end{align}
Adding \eqref{DCRthm:Eq9} and \eqref{DCRthm:Eq11} we obtain \eqref{DCRthm:funcrate}, which conlcudes our proof. 
%for all $ i =1,\ldots,n,$
%\begin{align}
%&F(z_i(t)\mathbf{1}) - f^*\leq \frac{2\Gamma_0(t)}{\sqrt{t}-1} +\frac{nV(\xbar(0))+C^2(\ln(t)+1)}{2(\sqrt{t}-1)},\nonumber
%\end{align}
%which is .
\end{proof}

%% file: appendix.tex
%!TEX root = Sigmetrics18_DBS.tex
\appendix

\section{Appendix}

\subsection{Extension to $\mathbb{R}^d$}\label{subsec:appendix_Rdextend}
We present here a sketch of key steps to extend our analysis for the case $d\geq 1$. In this section, use uppercase letters in boldface for matrices $\Xbf$ in $\Rset^{n\times d}$.  We now have $\xbf_i\in\Rset^{d}$ for all $i\in\Vcal$ and $f_i:\Rset^{d}\rightarrow\Rset$. We define the following notation 
\begin{align*}
&\Xbf = \left(\begin{array}{ccc}
\xbf_1^T\\
\ldots\\
\xbf_n^T
\end{array}\right)\in\Rset^{n\times d},\quad \xbarbf = \frac{1}{n}\sum_{i=1}^n \xbf_i \in\Rset^{d},\quad \Xbarbf = \left(\begin{array}{c}
\bar{\xbf}^T\\ 
\ldots\\ 
\bar{\xbf}^T\\
\end{array}\right)\in\Rset^{n\times d},\\
&F(\Xbf) \triangleq \sum_{i=1}^n f_i(\xbf_i),\quad \nabla F(\Xbf) = \left(\begin{array}{ccc}
\nabla f_1^T(\xbf_1)\\ 
\ldots\\
\nabla f_n^T(\xbf_n)\\
\end{array}\right)\in\Rset^{n\times d}.
\end{align*}
Given a matrix $\A$ we denote its $i-$th row as $\abf_i^T\in\Rset^{1\times n}$, i.e.,  
\begin{align*}
&\A = \left(\begin{array}{c}
\abf_1^T\\ 
\ldots\\
\abf_n^T
\end{array}\right)\in\Rset^{n\times d}.
\end{align*}
Moreover, we write $\|\A\|_F$ as the Frobenius norm of $\A$. With these notations the updates in \eqref{analysis:viUpdate}--\eqref{analysis:xbarUpdate} can be rewritten as 
\begin{align*}
&\Vbf(t)
= -\beta\Xbf(t) + \beta \A\Xbf(t-\tau) - \alpha(t) \nabla F(\Xbf(t)),\\
&\dot{\Xbf}(t)
= \Pcal_{\Tcal_{\Xcal(\Xbf(t))}}[\Xbf(t)] = \Vbf(t)-\zetabf(\Vbf(t)),\\
&\bar{\vbf}(t) = -\beta\bar{\xbf}(t) + \beta\bar{\xbf}(t-\tau) - \frac{\alpha(t)}{n}\sum_{i=1}^n \nabla f_i(\xbf_i(t))\\
&\dot{\bar{\xbf}}(t) = \bar{\vbf}(t) - \bar{\zetabf}(t),
\end{align*}
where the projection $\Pcal_{\Tcal_{\Xcal(\Xbf(t))}}[\Xbf(t)]$ is the row-wise projection. Finally, we use the following result studied in \cite{Nedic2010a}, which is a general version of Lemma \ref{lem_analysis:projection}, to analyze the impact of the projection.
\begin{lemma}[Lemma $1$ \cite{Nedic2010a}]\label{lem:Projection}
Let $\mathcal{X}$ be a nonempty closed convex set in $\mathbb{R}^d$. Then, we have for any $\xbf\in\mathbb{R}^d$
\begin{enumerate}
\item[(a)] $(\mathcal{P}_{\mathcal{X}}[\xbf]-\xbf)^T(\xbf-\ybf)\leq -\|\mathcal{P}_{\mathcal{X}}[\xbf]-\xbf\|_2^2\;$ for all $\ybf\in\mathcal{X}$
\item[(b)] $\|\mathcal{P}_{\mathcal{X}}[\xbf]-\ybf\|_2^2\leq \|\xbf-\ybf\|_2^2 -\|\mathcal{P}_{\mathcal{X}}[\xbf]-\xbf\|_2^2\;$ for all $\ybf\in\mathcal{X}$
\end{enumerate}
\end{lemma}
We now present the analysis for the general versions of Lemma \ref{lem:pertbAverdelay} and Theorem \ref{DCRthm:ConvRate}, which are given in the following two lemmas.

\begin{lemma}\label{lem_const_case:pertbAverdelay}
Suppose Assumptions \ref{assump:convexity}-- \ref{assump:doublystochastic} hold. Let the trajectories of $\xbf_i(t)$ be updated by Algorithm \ref{alg:delay}. Let $\{\alpha(t)\}$ be a given positive scalar sequence with $\alpha(0) = 1$. Moreover, let $\beta \in (0,\frac{\ln(1/\sigma_2)}{\tau})$ and $\gamma =\sigma_2 e^{\beta\tau} \in (0,1)$. Then 
\begin{itemize}
\item[(1)] For all $t\geq 0$  we have
\begin{align}
\|\Xbf(t)-\bar{\Xbf}(t)\|_F \leq \mu(t) + \beta\sigma_2\int_{0}^{t}e^{-\beta(1-\gamma)(t-u)}\mu(u-\tau)du,\label{lem_const_case:AverIneq}
\end{align}
where
\begin{align}
&\mu(t) =  e \frac{\|\Xbf(0)\|_F+2C}{\beta}e^{-\beta t/2} + \frac{2C\alpha(t/2)}{\beta}.\label{lem_const_case:lambda}
\end{align}
\item[(2)]
If $\{\alpha(t)\}$  is a non-increasing positive scalar sequence such that $\lim_{t\rightarrow\infty}\alpha(t) = 0$ then we have
\begin{align}
\lim_{t\rightarrow\infty} \|\xbf_i(t) - \bar{\xbf}(t)\|_2 = 0\quad \text{for all } i=1,2\ldots,n.\label{lem_const_case:AsympConv}
\end{align}
\item[(3)] Further we have
\begin{align}
\int_{0}^t \alpha(u)\|\Xbf(u)-\bar{\Xbf}(u)\|_2du\leq \frac{8\left(\|\Xbf(0)\|_F+2C\right)e^{\beta\tau/2}}{\beta^3(1-\gamma)^2}+ \frac{4C}{\beta^2(1-\gamma)}\int_{0}^{t}\alpha^2(\gamma u/4-\tau)du.\label{lem_const_case:Upperbound}
\end{align}
\end{itemize}
\end{lemma}

\begin{proof}[Proof sketch]
As mentioned, the key step in the proof of Lemma \ref{lem_const_case:pertbAverdelay} is to show \eqref{lem_const_case:AverIneq}. The analysis of \eqref{lem_const_case:AsympConv} and \eqref{lem_const_case:Upperbound} are consequences of \eqref{lem_const_case:AverIneq}. Consider the following notation:
\begin{align*}
\Gbf(t)&=(\I-\frac{1}{n}\mathbf{1}\mathbf{1}^T)\nabla F(\Xbf(t)),\quad\quad \Hbf(t) = \left(\I-\frac{1}{n}\1\1^T\right)\zetabf(\Vbf(t))\quad\quad \Ybf(t) = \Xbf(t) - \bar{\Xbf}(t).
\end{align*} 
We first consider 
\begin{align*}
\dot{\ybf}_i(t) &= \dot{\xbf}_i(t) - \dot{\bar{\xbf}}(t)\nonumber\\
&= -\beta \xbf_i(t) +\beta\sum_{j=1}^n a_{ij}\xbf_j(t-\tau) - \alpha(t)\nabla f_i(\xbf_i(t)) -\zetabf_i(t)\nonumber\\ 
&\quad + \beta \bar{\xbf}(t) - \beta\bar{\xbf}(t-\tau) + \frac{\alpha(t)}{n}\sum_{j=1}^n\nabla f_j(\xbf_j(t))+\bar{\zetabf}(t)\nonumber\\
&= -\ybf_i(t) + \beta\sum_{j=1}^n a_{ij}\ybf_j(t-\tau) - \alpha \gbf_i(t)  - \hbf_i(t),
\end{align*}
which implies 
\begin{align*}
\ybf_i(t) = e^{-t}\ybf_i(0) + \int_0^t e^{-(t-u)}\left( \beta\sum_{j=1}^n a_{ij}\ybf_j(u-\tau) - \alpha \gbf_i(u)  - \hbf_i(u)\right)dt.
\end{align*}
Thus we obtain
\begin{align}
\Ybf(t)&= e^{-\beta t}\Ybf(0)+ \beta \int_{0}^{t}e^{-\beta(t-u)}\A \Ybf(u-\tau)du - \int_{0}^{t}e^{-\beta(t-u)}\left(\alpha(u)\Gbf(u)+\Hbf(u)\right)du.\label{lem_const_case:Eq1a}
\end{align}
Recall that $\Ybf(t) = \Xbf(t) - \bar{\Xbf}(t) = (\I-\frac{1}{n}\1\1^T)\Xbf(t)$. In addition, note that $\1^T\Ybf(t) = \1^T(\I-\frac{1}{n}\1\1^T)\Xbf(t) = \mathbf{0}$, implying that each column of $\Ybf(t)\notin span\{\1\}$. Indeed, if there exists at least one column of $\Ybf(t)$, namely, $\mathbf{p}_{\ell}(t)$, such that $\mathbf{p}_{\ell}(t)\in span\{\1\}$ then $\1^T\mathbf{p}_{\ell}(t) \neq 0$ but $\1^T\Ybf(t) = \mathbf{0}$, a contradiction. The previous observation implies that 
\begin{align}
\|\A\Ybf(t)\|_F^2 = \sum_{i=1}^n\|\A\mathbf{p}_{i}(t)\|_2^2 \leq \sum_{i=1}^n \sigma_2\|\mathbf{p}_{i}(t)\|_2^2  = \sigma_2 \|\Ybf\|_{F}^2,\label{lem_const_case:Eq1b}
\end{align}  
where $\mathbf{p}_i(t)$ are columns of $\Ybf(t)$. Taking the Frobenius norm on both sides of \eqref{lem_const_case:Eq1a}, and using \eqref{dplem:gbound} and \eqref{lem_const_case:Eq1b} we have
\begin{align}
\|\Ybf(t)\|_{F} &\leq e^{-\beta t}\|\Ybf(0)\|_F + \beta\sigma_2 \int_{0}^{t}e^{-\beta(t-u)}\|\Ybf(u-\tau)\|_F du+C\int_{0}^{t}e^{-\beta(t-u)}\alpha(u)du\nonumber\\
&\quad + \int_{0}^{t}e^{-\beta(t-u)}\|\zetabf(u)\|_Fdu.\label{lem_const_case:Eq1}
\end{align}
We now use Lemma \ref{lem:Projection} to construct an upper bound on the last term on the right hand side of \eqref{lem_const_case:Eq1}. First, since $\A$ is doubly stochastic and $\xbf_j(t-\tau)\in \mathcal{X}$ $\forall j$ we have $\sum_{j\in\mathcal{N}_i}a_{ij}\xbf_j(t-\tau)\in\Xcal$. Thus, by \eqref{def:feasible_set} with $\theta = \beta^{-1}$ we have
\begin{align*}
r_i(t) = -\beta \xbf_i(t) + \beta \sum_{j\in\mathcal{N}_i}a_{ij}\xbf_j(t-\tau)\in\mathcal{D}_{\Xcal}(\xbf_i(t)).
\end{align*}
Hence, by  Proposition \ref{prop:TangentCone} we have $r_i(t)\in \mathcal{T}_{\mathcal{X}}(\xbf_i(t))$.  By Lemma \ref{lem:Projection}(b), we have 
\begin{align}
\|\mathcal{P}_{\mathcal{T}_{\mathcal{X}}(\xbf_i(t))}[\vbf_i(t)]-\mathbf{r}_i(t)\|_2^2\leq \|\vbf_i(t) - \mathbf{r}_i(t)\|_2^2 - \|\mathcal{P}_{\mathcal{T}_{\mathcal{X}}(\xbf_i(t))}[\vbf_i(t)] - \vbf_i(t) |_2^2,\nonumber  
\end{align}
which since ${\zetabf_i(t) = \vbf_i(t) - \mathcal{P}_{\mathcal{T}_{\mathcal{X}}(\xbf_i(t))}[\vbf_i(t)]}$ implies
\begin{equation}
\|\zetabf_i(t)\|_2 \leq \|\vbf_i(t) - \mathbf{r}_i(t)\|_2 = \|\alpha(t)\nabla f_i(\xbf_i(t))\|_2 \leq C_i\alpha(t).\label{lem_const_case:eibound}
\end{equation}
Thus we obtain $\|\zetabf(t) - \bar{\zetabf}(t)\|_F = \|\left(\I-\frac{1}{n}\1\1^T\right)\zetabf(t)\|_F \leq \|\zetabf(t)\|_F \leq C\alpha(t)$. Substituting the previous relation into \eqref{lem_const_case:Eq1} and using \eqref{dplem:Eq3d} we obtain \eqref{lem_const_case:AverIneq}.
\end{proof}

In the lemma below, with some abuse of notation we denote by $\Xbf_i(t)$ the matrix whose all the rows are $\xbf_i^T(t)$, i.e., 
\begin{align*}
\Xbf_i(t) = \left( \begin{array}{c}
\xbf_i^T(t)\\
\ldots\\
\xbf_i^T(t)
\end{array}\right) 
\end{align*}

\begin{lemma}\label{lem_const_case:ConvRate}
Suppose Assumptions \ref{assump:convexity}--\ref{assump:doublystochastic} hold. Let the trajectories of $\xbf_i(t)$ be updated by Algorithm \ref{alg:delay}. Let $\beta \in (0,\frac{\ln(1/\sigma_2)}{\tau})$ and $\gamma =\sigma_2 e^{\beta\tau} \in (0,1)$. Let $\{\alpha(t)\}$ be a given positive scalar sequence such that $\alpha(t) = 1/\sqrt{t}$ for $t\geq 1$ and $\alpha(t) = 1$ for $t\leq 1$. Then for each $i=1,\ldots,n$ we have\vspace{-0.3cm}

\begin{align}
&F\left(\frac{\int_{0}^t\alpha(u)\Xbf_i(u)du}{\int_{0}^t\alpha(u)du}\right) -f^*\leq \frac{2\Gamma_0(t)+nV(\xbar(0))}{2(\sqrt{t}-1)},
\label{lem_const_case:funcrate}
\end{align}
where,
\begin{align}
\Gamma_0(t) \triangleq  &  \frac{24C\left(\|\Xbf(0)\|_F+2C\right)e^{\beta\tau/2}}{\beta^3(1-\gamma)^2} +  \frac{48C^2(1+\tau)}{\beta^2\gamma(1-\gamma)}+C^2\ln(t)+ \frac{48C^2\ln(\gamma t-4\tau)}{\beta^2\gamma(1-\gamma)}\cdot
\end{align}
\end{lemma}

\begin{proof}[Proof Sketch]
Let $\xbf^*$ be a solution of problem \eqref{prob:obj}. Consider the candidate Razumikhin-Krasovskii Lyapunov function given in \eqref{DCRthm:Lyapunov}, where its derivative is given as
\begin{align}
\dot{V}(\bar{\xbf}(t))&\leq \underbrace{-\frac{\alpha(t)}{n}\sum_{i=1}^n (\bar{\xbf}(t)-\xbf^*)^T \nabla f_i(\xbf_i(t))}_{W_1} \underbrace{-\frac{1}{n}\sum_{i=1}^n (\bar{\xbf}(t)-\xbf^*)^T \zetabf_i(\xbf_i(t))}_{W_2}\nonumber\\ 
&\leq W_1+W_2.\label{lem_const_case:Eq2}
\end{align}
The term $W_1$ can be upper bounded by using \eqref{DCRthm:Eq1c}. Here we focus on delivering the upper bound of $W_2$. Recall that $\zetabf_i(t) = \vbf_i(t) - \mathcal{P}_{\Tcal_{\Xcal(\xbf_i(t))}}[\vbf_i(t)]$. Consider
\begin{align}
&W_2= -(\bar{\xbf}(t)-\xbf^*)\bar{\zetabf}(t)\nonumber\\
&= -\frac{1}{n} \sum_{i=1}^n\left(\bar{\xbf}(t)-(1+\beta)\xbf_i(t) + \beta\sum_{j=1}^n a_{ij}\xbf_j(t-\tau) - \vbf_i(t)\right)^T\zetabf_i(t)\nonumber\\ 
&\quad - \frac{1}{n} \sum_{i=1}^n\left(\vbf_i(t) + (1+\beta)\xbf_i(t) - \beta\sum_{j=1}^n a_{ij}\xbf_j(t-\tau)-\xbf^*\right)^T\zetabf_i(t), \label{lem_const_case:Eq2a}
\end{align}
where by \eqref{analysis:viUpdate} the first sum is equivalent to 
\begin{align*}
&-\frac{1}{n} \sum_{i=1}^n\left(\bar{\xbf}(t)-(1+\beta)\xbf_i(t) + \beta\sum_{j=1}^n a_{ij}\xbf_j(t-\tau) - \vbf_i(t)\right)^T\zetabf_i(t)\nonumber\\
& = -\frac{1}{n} \sum_{i=1}^n\left(\bar{\xbf}(t)-x_i(t)+\alpha(t)\nabla f_i(\xbf_i(t))\right)^T\zetabf_i(t)\nonumber\\
&\leq \frac{1}{n}\sum_{i=1}^n\|\bar{\xbf}(t) - \xbf_i(t)\|_2 \|\zetabf_i(t)\|_2 + \frac{1}{n} \sum_{i=1}^n\alpha(t)\|\nabla f_i(\xbf_i(t))\|_2\|\zetabf_i(t)\|_2\nonumber\\
&\overset{\eqref{lem_const_case:eibound}}{\leq} \frac{C \alpha(t)}{n}\|\Xbf(t) - \bar{\Xbf}(t)\mathbf{1}\|_F + \frac{C^2\alpha^2(t)}{n}.
\end{align*}
On the other hand, let $\mathbf{r}_i(t)$ be defined as
\begin{align*}
\mathbf{r}_i(t) = \xbf^* - (1+\beta)\xbf_i(t) + \beta\sum_{j=1}^n a_{ij}\xbf_j(t-\tau).
\end{align*}
Consider
\begin{align*}
\xbf_i(t) + \frac{1}{2}\mathbf{r}_i(t) = \frac{1-\beta}{2}\xbf_i(t) + \frac{1}{2}\xbf^* + \frac{\beta}{2}\sum_{j=1}^n a_{ij}\xbf_j(t-\tau)\in\Xcal.
\end{align*}
which by \eqref{def:feasible_set} with $\theta = 1/2$ implies $\mathbf{r}_i(t)\in \mathcal{D}_{\Xcal}(\xbf_i(t))$. In addition, by Proposition \ref{prop:TangentCone} we have $\mathbf{r}_i(t)\in\Tcal_{\Xcal(\xbf_i(t))}$. Thus, by applying (1a) in Lemma \ref{lem:Projection} to the second term in \eqref{lem_const_case:Eq1a} we obtain
{\small
\begin{align*}
 &- \frac{1}{n}\sum_{i=1}^n\left(\vbf_i(t)-\mathbf{r}_i(t)\right)^T\zetabf_i(t) 
\leq - \frac{1}{n}\sum_{i=1}^n \left\|\vbf_i(t) -  \mathcal{P}_{\Tcal_{\Xcal(\xbf_i(t))}}[\vbf_i(t)]\right\|_2^2 = -\frac{1}{n}\|\zetabf(t)\|_F^2.
\end{align*}}
Applying the preceding two relations into \eqref{lem_const_case:Eq1a} we obtain
\begin{align}
W_2 \leq \frac{C \alpha(t)}{n}\|\Xbf(t) - \bar{\Xbf}(t)\mathbf{1}\|_{F} + \frac{C^2\alpha^2(t)}{n} - \frac{1}{n}\|\zetabf(t)\|_F^2\leq \frac{C \alpha(t)}{n}\|\Xbf(t) - \bar{\Xbf}(t)\mathbf{1}\|_{2} + \frac{C^2\alpha^2(t)}{n}. \label{lem_const_case:Eq2b}
\end{align}
Thus we obtain the same result as in \eqref{DCRthm:Eq4}, i.e., 
\begin{align*}
\dot{V}(\bar{x}(t)) \leq \frac{3\alpha(t)C}{n}\|\Xbf(t)-\bar{\Xbf}(t)\|_F + \frac{C^2\alpha^2(t)}{n}- \frac{\alpha(t)}{n}(F(\bar{\Xbf}(t)\mathbf{1})-f^*).
\end{align*}
The rest of this proof is the same as the one in Section \ref{subsec:DelayProofs_rate}.
\end{proof}